\documentclass{amsart}
\usepackage{graphicx} % Required for inserting images

\usepackage{amssymb,amsmath,amsthm,latexsym,mathtools,dsfont}
\usepackage{accents}
\usepackage{mathrsfs}
\usepackage[mathscr]{euscript}
\usepackage{mathtools}
\usepackage{mdwlist}
\usepackage{graphicx}
\usepackage{hyperref}
\usepackage{todonotes}
\usepackage{enumerate}
\usepackage[shortlabels]{enumitem}
\usepackage{a4wide}
\usepackage{bbm}

\usepackage{physics}%nicely handles subscripts at absolute value and modulus signs (when using the abs and norm commands)
\usepackage{color}
\definecolor{amber}{rgb}{1.0, 0.75, 0.0}

\newcommand{\on}{\operatorname}

\newcommand{\IN}{\mathbb{N}}

\newcommand{\IK}{\mathbb{K}}

\newcommand{\I}{\mathcal{I}}
\newcommand{\J}{\mathcal{J}}

\newcommand{\mf}{\mathfrak}
\newcommand{\Fin}{\mathrm{Fin}}
\newcommand{\Exh}{\mathrm{Exh}} 
\newcommand{\supp}{\mathrm{supp}} 
\newcommand{\CR}{\mathcal{CR}}
\newcommand{\w}{\mathbbm}

\newtheorem{theorem}{Theorem}[section]
\newtheorem{definition}[theorem]{Definition}
\newtheorem{lemma}[theorem]{Lemma}
\newtheorem{example}[theorem]{Example}
\newtheorem{proposition}[theorem]{Proposition}
\newtheorem{remark}[theorem]{Remark}
\newtheorem{corollary}[theorem]{Corollary}
\newtheorem{question}[theorem]{Question}

\title{Zoo of ideal Schauder bases}

\author{Adam Kwela}
\address[Adam Kwela]{Institute of Mathematics\\ Faculty of Mathematics\\ Physics and Informatics\\ University of Gda\'{n}sk\\ ul.~Wita Stwosza 57\\ 80-308 Gda\'{n}sk\\ Poland}
\email{Adam.Kwela@ug.edu.pl}
\urladdr{https://mat.ug.edu.pl/~akwela}

\author{Jaros{\l}aw Swaczyna}
\address[Jaros{\l}aw Swaczyna]{Institute of Mathematics, {\L}\'od\'z University of Technology, Aleje Politechniki 8, 93-590 {\L}\'od\'z, Poland}
\email{jaroslaw.swaczyna@p.lodz.pl}
\date{December 2024}
\thanks{The second-named author acknowledges with thanks funding received from NCN project SONATA BIS 13 No. 2023/50/E/ST1/00067}

\begin{document}

\begin{abstract}

We investigate the notion of filter (equivalently: ideal) Schauder basis of a Banach space. We do so by providing bunch of new examples of such bases that are not the standard ones, especially within classical Banach spaces ($\ell_p$, $c_0$, James' space). Those examples lead to distinguishing and characterizing filters (equivalently: ideals) in terms of Schauder bases. We investigate the relationship between possibly basic sequences and filters (equivalently: ideals) on the set of natural numbers. 
\end{abstract}

\maketitle

\section{Preliminaries}
We aim to keep our paper readable for both set theorists and functional analysts; therefore, we recall some of the basic definitions from both fields.

Throughout the paper, we use standard set-theoretic notation and identify each natural number $n\in\omega=\{0,1,2,\ldots\}$ with the set $\{0,1,\ldots,n-1\}$.

\subsection{Ideals}
We will be concerned with ideals $\I$ on the set of natural numbers $\omega$. By an ideal on a countable set $M$ we mean a family $\I \subseteq \mathcal{P}(M)$ such that:
\begin{itemize}
    \item $\I$ contains all finite subsets of $M$;
    \item if $A,B \in \I$, then also $A\cup B \in \I$;
    \item if $A \in \I$ and $B \subseteq A$, then also $B \in I$.
\end{itemize}
We will focus primarily on non-trivial ideals, that is, ideals $\I$ such that $\I \neq \mathcal{P}(M)$, and our ideals will typically be defined on $M=\omega$. 

\begin{remark}
Clearly, the notion of ideal is dual to the notion of filter: $\I$ is an ideal if and only if the family of complements of elements of $\I$ forms a filter. Therefore, choice between working in terms of ideal bases and filter bases is purely aesthetic.
\end{remark}    

We treat the power set $\mathcal{P}(M)$ as the space $2^M$ of all functions $f:M\rightarrow 2$ (equipped with the product topology, where each space $2=\left\{0,1\right\}$ carries the discrete topology) by identifying subsets of $M$ with their characteristic functions. All topological and descriptive notions in the context of ideals will refer to this topology. 

An ideal $\I$ on $M$ is tall if for any infinite $A\subseteq M$ there is an infinite $B\subseteq A$ such that $B\in\I$. Equivalently, $\I$ is tall if $\I\restriction A\neq[A]^{<\omega}$, for any $A\notin\I$, where
\[
\I\restriction A=\left\{B\cap A:B\in\I\right\}.
\]
We say that an ideal $\I$ on $M$ is generated by a family $\mathcal{A}\subseteq\mathcal{P}(M)$ if 
\[
\I=\{B\subseteq M:\exists_{n\in\omega}\ \exists_{A_0,\ldots,A_n\in\mathcal{A}}\ B\subseteq A_0\cup\ldots\cup A_n\}.
\]
An ideal $\I$ is called a P-ideal if for each sequence $(A_n)\subseteq  \I$ there is a set $A\in\I$ such that $A_n\setminus A$ is finite for all $n\in\omega$. By a result of Solecki, every analytic P-ideal is $\bf{F_{\sigma\delta}}$ (see \cite{Solecki}).

If $\{X_m: m\in M\}$ is a family of sets, then $\sum_{m\in M}X_m=\{(m,x): m\in M,x\in X_m\}$ is its disjoint sum. The vertical section of a set $A\subseteq \sum_{m\in M}X_m$ at a point $m\in M$ is defined
by $A_{(m)} = \{x\in X_m: (m,x)\in A\}$. Let $\I$ and $\J$ be ideals on $M$ and $N$, respectively, and $\I_m$, for each $m\in M$, be an ideal on $X_m$. We define the following new ideals:
\begin{itemize}
	\item $\sum_{m\in M}\I_m=\{A\subseteq \sum_{m\in M}X_m: \forall_{m\in M}\ A_{(m)}\in\I_m\}$;
 	%\item $\sum^J_{t\in T}\I_t=\{M\subseteq \sum_{t\in T}X_t:\ \{t\in T:\ M_{(t)}\notin\I_t\}\in\J\}$;
    %\item $\I\otimes\J=\sum^\I_{s\in S}\J=\{M\subseteq S\times T:\ \{s\in S:\ M_{(s)}\notin\J\}\in\I\}$;
	%\item $\I\otimes \{\emptyset\} = \{M\subseteq S\times \omega:\ \{s\in S:\ M_{(s)}\ne\emptyset\}\in \I\}$;
	\item $\{\emptyset\} \otimes  \I= \{A\subseteq \omega\times M: A_{(n)} \in \I\text{ for all }n\in \omega\}$;
    \item $\I\oplus\J=\{A\subseteq(\{0\}\times M)\cup(\{1\}\times N): A_{(0)}\in\I\text{ and }A_{(1)}\in\J\}$.
\end{itemize}

Two ideals $\I$ and $\J$ on $M$ and $N$, respectively, are isomorphic ($\I\cong\J$) if there is $f:N\to M$ such that:
\[
f^{-1}[A]\in\J\ \Leftrightarrow\ A\in\I,
\]
for all $A\subseteq M$. Two isomorphic ideals have the same descriptive complexity (see for instance \cite[Lemma 7.2]{Debs}). 

\begin{lemma}[Folklore]
\label{lem-isomorphisms}
Let $\I$ be an ideal.
\begin{itemize}
    \item[(a)] If $\I$ is not tall, then $\I\cong\I\oplus\Fin$.
    \item[(b)] If $\I\neq\Fin$, then $\I\cong\I\oplus\mathcal{P}(\omega)$.
\end{itemize}
\end{lemma}

\begin{proof}
Without loss of generality, we may assume that $\I$ is an ideal on $\omega$. 

Since $\I$ is not tall ($\I\neq\Fin$, respectively), there is $A\notin\I$ such that $\I\restriction A=\Fin\restriction A$ (an infinite $A\in\I$, respectively). Let $g:((\{0\}\times A)\cup(\{1\}\times\omega))\to A$ be any bijection. Then $f:2\times\omega\to\omega$ given by:
\[
f(i,n)=\begin{cases}
    g(i,n), & \text{if }(i,n)\in(\{0\}\times A)\cup(\{1\}\times\omega),\cr
    n, & \text{otherwise,}
\end{cases}
\]
for all $(i,n)\in 2\times\omega$, witnesses $\I\cong\I\oplus\Fin$ ($\I\cong\I\oplus\mathcal{P}(\omega)$, respectively).
\end{proof}

\begin{example}
The following is a list of some well-known ideals:
\begin{itemize}
    \item The smallest ideal is $\Fin=[\omega]^{<\omega}$.  It is an $\bf{F_\sigma}$ non-tall P-ideal.
    \item $\Fin\oplus\mathcal{P}(\omega)$ is an ideal on $2\times\omega$ consisting of all $A\subseteq 2\times\omega$ such that $A\cap(\{0\}\times\omega)$ is finite. It is an $\bf{F_\sigma}$ non-tall P-ideal. 
    \item A summable ideal is an ideal of the form
\[
\I_{\w{r}}=\left\{A\subseteq\omega:\sum_{n\in A}r_n<\infty\right\},
\]
where $\w{r}=(r_n)$ is a sequence of positive reals such that $\sum_{n\in\omega}r_n=\infty$. Each summable ideal is an $\bf{F_\sigma}$ P-ideal. The most well-known summable ideal is
\[
\I_{1/n}=\left\{A\subseteq\omega:\sum_{n\in A}\frac{1}{n+1}<\infty\right\}.
\]
It is a tall ideal. Note that $\Fin$ is a summable ideal (given, eg. by any constant sequence $\w{r}$) and $\Fin\oplus\mathcal{P}(\omega)$ is isomorphic to the summable ideal $\{A\subseteq\omega: A\cap\{2n:n\in\omega\}\in\Fin\}$ given by the sequence:
\[
r_n=\begin{cases}
    \frac{1}{2^n} & \text{if }n\text{ is odd},\cr
    1 & \text{if }n\text{ is even}.
\end{cases}
\]
\item The ideal of asymptotic density zero is given by:
\[
\I_d= \left\{A \subseteq \omega: \frac{|A \cap (n+1)|}{n+1} \to 0 \right\}
\]
(here we use the standard set-theoretic convention and treat a number $n\in\omega$ as the set $\{0,1,\ldots,n-1\}$). It is an $\bf{F_{\sigma\delta}}$, but not $\bf{F_{\sigma}}$, tall P-ideal. 
\item $\{\emptyset\}\otimes\Fin$ is an ideal on $\omega^2$ consisting of all $A\subseteq\omega^2$ such that $A_{(n)}=\{k\in\omega: (n,k)\in A\}\in\Fin$ for all $n\in\omega$.  It is an $\bf{F_{\sigma\delta}}$, but not $\bf{F_{\sigma}}$, non-tall P-ideal. 
\end{itemize}
For more examples of ideals see for instance \cite{Farah}.
\end{example}

\subsection{Submeasures}

A function $\phi:\mathcal{P}(\omega)\to[0,\infty]$ is called a submeasure if $\phi(\emptyset)=0$, $\phi(\{n\})<\infty$, for each $n\in\omega$, and 
$$\phi(A)\leq\phi(A\cup B)\leq\phi(A)+\phi(B)$$
for all $A,B\subseteq\omega$. A submeasure $\phi$ is lower semicontinuous (lsc, in short) if $\phi(A)=\lim_{n\to\infty}\phi(A\cap n)$ for all $A\subseteq\omega$. The support of a submeasure $\phi$ is defined as $\supp(\phi)=\{n\in\omega: \phi(\{n\})\neq 0\}$.

Mazur in \cite[Lemma 1.2]{Mazur} proved that an ideal on $\omega$ is $\bf{F_\sigma}$ if and only if it is of the form:
$$\Fin(\phi)=\left\{A\subseteq\omega: \phi(A)<\infty\right\}$$
for some lower semicontinuous submeasure $\phi$ such that $\omega\notin\Fin(\phi)$ (see also \cite[Theorem 1.2.5]{Farah}). Similarly, Solecki in \cite[Theorem 3.1]{SoleckiExh} showed that an ideal on $\omega$ is an analytic P-ideal if and only if it is of the form:
$$\Exh(\phi)=\left\{A\subseteq\omega: \lim_{n\to\infty}\phi(A\setminus n)=0\right\}$$
for some lower semicontinuous submeasure $\phi$ such that $\omega\notin\Exh(\phi)$ (see also \cite[Theorem 1.2.5]{Farah}). 

A submeasure $\phi$ is non-pathological, if
$$\phi(A)=\sup\{\mu(A): \mu\text{ is a measure with }\supp(\mu)\in\Fin\text{ and }\mu(B)\leq\phi(B)\text{ for all }B\subseteq\omega\},$$
for all $A\subseteq\omega$. We say that an $\bf{F_\sigma}$ ideal (analytic P-ideal) is non-pathological, if it is of the form $\Fin(\phi)$ ($\Exh(\phi)$, respectively), for some non-pathological submeasure $\phi$.

\subsection{Ideal convergence}

Following \cite{KSW}, for an ideal $\I$ on $\omega$ and a sequence $(x_n)$ of elements of some metric space $(X,d)$, we say that $(x_n)$ is $\I$-convergent to some $x\in X$ if almost all $x_n$'s in the sense of $\I$, are arbitrarily close to $x$, that is, for every $\varepsilon>0$ we have $\{n \in \omega: d(x_n , x)> \varepsilon \} \in \I$. In such a case, we write $x_n \to_\I x$ or $\lim_{n,\I}x_n=x$. Note that for the particular ideal of asymptotic density zero, ideal convergence has been previously considered under the name of statistical convergence (see e.g. \cite{Fast}, \cite{Fridy}, \cite{Salat}, \cite{Steinhaus}).

Similarly, for a Banach space $X$, we say that a series $(\w{x}_n)$ is $\I$-convergent to some $\w{x}\in X$, if the sequence $(\w{a}_n)$ of partial sums, given by $\w{a}_n = \sum_{i=0}^n \w{x}_n$, is $\I$-convergent to $X$. In such a case we write 
\[
\w{x}=\lim_{n,\I} \w{a}_n = \sum_{n,\I} \w{x}_n.
\]
Note that it means "almost all partial sums are close to $\w{x}$" rather than "sums of almost all elements are close to $\w{x}$". In fact, even though the second statement might be more intuitive on the first thought, it does not make sense at all. 

\subsection{Schauder bases}

Given a Banach space $X$ over the real or complex field $\IK$, by a Schauder basis of $X$ we usually mean such a sequence $(\w{u}_n)\subseteq X$ that for any $\w{x}\in X$ there exists exactly one sequence of scalars $(\alpha_n)\subseteq \IK$ such that $\w{x}= \sum_{n=0}^{\infty} \alpha_n \w{u}_n$. In such a case we can associate with it a sequence of functionals $\w{u}_n^\star \colon X \to \IK $ such that $\w{u}_i^\star \colon \sum_{n=0}^{\infty} \alpha_n \w{u}_n \mapsto \alpha_i$. From Banach Space Theory, it is a standard yet nontrivial fact that all $\w{u}_n^\star$'s are continuous and linear (see eg. \cite[Theorem 1.1.3]{AK}). Throughout the paper we will be dealing with specific examples of bases and we find fixing coordinate functionals convenient, so by a Schauder basis of $X$ we mean the sequence of pairs $\hat{\w{u}}=(\w{u}_n,\w{u_n^\star})$.

A Schauder basis $\hat{\w{u}}=(\w{u}_n,\w{u_n^\star})$ of a Banach space $X$ is unconditional, if the series $\sum_{n=0}^{\infty} \w{u}_n^\star(\w{x}) \w{u}_n$ converges unconditionally, for every $\w{x}\in X$.

Throughout the whole paper by $\hat{\w{e}}=(\w{e}_n,\w{e_n^\star})$ we will denote the standard Schauder basis of $c_0$. Note that $\hat{\w{e}}$ is also a Schauder basis of each $\ell_p$, $1\leq p<\infty$. It is known that $\hat{\w{e}}$ is unconditional in $c_0$ and in each $\ell_p$, $1\leq p<\infty$.

\begin{lemma}[Folklore]
\label{lem-monotonicznosc}
Let $\hat{\w{u}}$ be an unconditional Schauder base of $X$. If $\sum_n a_n \w{u}_n\in X$ and $(b_n)\subseteq \IK$ is such that $|b_n|\leq |a_n|$, for every $n \in \omega$, then $\sum_n b_n \w{u}_n \in X$.
\end{lemma}

\begin{proof}
Since $\hat{\w{u}}$ is unconditional, by \cite[Proposition 3.1.3]{AK}, there is $K\geq 1$ such that
    \[
    \left\| \sum_{i=0}^n c_i \w{u}_i \right\|_X \leq K \left\| \sum_{i=0}^n d_i \w{u}_i \right\|_X,
\]
for all $n\in\omega$ and scalars $c_0,\ldots,c_n,d_0,\ldots,d_n$ such that $|c_i|\leq |d_i|$, for $i\leq n$.

Fix $n\in\omega$. Then for any $m>n$ 
    \[
    \left\| \sum_{i=n}^m b_n \w{u}_n \right\|_X \leq K \left\| \sum_{i=n}^m a_n \w{u}_n \right\|_X \leq  K^2 \left\| \sum_{i=n}^\infty a_n \w{u}_n \right\|_X \xrightarrow{n \to \infty} 0.
\]
Therefore $ \sum_{n=0}^\infty b_n \w{u}_n  $ is convergent.
\end{proof}

\section{Ideal Schauder bases}

\subsection{Definition}

Natural way for considering ideal version of Schauder bases is to simply demand $\w{x}=\sum_{n,\I} \alpha_n \w{e}_n$ instead of $\w{x}=\sum_{n} \alpha_n \w{e}_n$ in which case linear coordinate functionals also appear, yet their continuity is not so clear anymore (see \cite{CGK,RKS,KS,Kochanek}  for more detailed discussion of this issue). Since throughout the presented paper we find fixing coordinate functionals convenient, therefore we consider the following definition.

If $X$ is a Banach space and $\hat{\w{a}}=(\w{a}_n,\w{a}^\star_n)\subseteq X\times \IK^X$ are such that: 
\[
\w{a}^\star_n(\w{a}_m)=\delta_{nm}=\begin{cases}
    1, & \text{if }n=m,\\
    0, & \text{if }n\neq m,
\end{cases}
\]
(which represents equality $\w{a}_n=\w{a}_n +\sum_{m\neq n}0\cdot \w{a}_m$), then we denote: 
\[
S_n(\hat{\w{a}})(\w{x})=\sum_{i=0}^n \w{a}_i^\star(\w{x}) \w{a}_i,
\]
for all $\w{x}\in X$ and $n\in\omega$. For an ideal $\I$ on $\omega$, we say that $\hat{\w{a}}$ is an $\I$-Schauder basis of $X$ if
\[
\{n\in\omega: \|\w{x}-S_n(\hat{\w{a}})(\w{x})\|_X\geq\varepsilon\}\in\I,
\]
that is $\w{x}=\sum_{n,\I} \w{a}^\star_n(\w{x})\w{a}_n$, for every $\w{x}\in X$ and $\varepsilon>0$.

It is a natural question if our definition is equivalent to the original one, i.e. if the sequence $(\w{a}_n^\star (\w{x}))$ is the only possible sequence of coordinates. Now we will explain why uniqueness of coefficients is not a big deal in the context of presented paper (however, note that in general there might be an issue if some of $\w{a}_n^\star$'s would not be continuous). Our reasoning is the same as in the case of standard bases (see eg. \cite[Definition 1.1.2, Theorem 1.1.3]{AK}).

\begin{lemma}
\label{L:wspolczynniki}
Let $(\w{a}_n)$ be such a sequence in some Banach space $X$ that $\w{a}_n\notin \overline{\on{span}\{\w{a}_m:m \neq n\}}$ for every $n \in \omega$. Fix two sequences of scalars $(\alpha_n)$ and $(\beta_n)$. Then for every $l\in \omega$ such that $\alpha_l\neq \beta_l$ there exists $\varepsilon>0$ such that $\|\sum_{i=0}^k \alpha_i\w{a}_i - \sum_{i=0}^{k'} \beta_i\w{a}_i\|_X>\varepsilon$ for every $k,k'>l$.
\end{lemma}

\begin{proof}
Note that $\sum_{i=0}^k \alpha_i\w{a}_i - \sum_{i=0}^{k'} \beta_i\w{a}_i=(\alpha_l-\beta_l)\w{a}_l-\w{b}$ for some $\w{b}\in \on{span}\{\w{a}_m:m \neq l\}$. Thus
\[
\left\|\sum_{i=0}^k \alpha_i\w{a}_i - \sum_{i=0}^{k'} \beta_i\w{a}_i\right\|_X\geq \on{dist}\left((\alpha_l-\beta_l)\w{a}_l, \overline{\on{span}\{\w{a}_m:m \neq l\}}\right)>0.
\]
Note that the distance above does not depend on $k$ and $k'$. 
\end{proof}

\begin{proposition}
\label{P:wspolczynniki}
Let $\hat{\w{a}}=(\w{a}_n,\w{a}^\star_n)\subseteq X\times X^\star$ be such that $\w{a}^\star_n(\w{a}_m)=\delta_{nm}$ for all $n,m\in\omega$. Fix $\w{x}\in X$ and a sequence of scalars $(\beta_n)$. Assume there are two nontrivial ideals $\I$ and $\J$ such that:
\[
\w{x}=\sum_{n,\I} \w{a}^\star_n(\w{x}) \w{a}_n= \sum_{n,\J} \beta_n \w{a}_n.
\]
Then $\w{a}_n^\star(x)=\beta_n$ for every $n\in \omega$. 
\end{proposition}

\begin{proof}
    Note that if $n\neq m$ then the condition $\w{a}^\star_n(\w{a}_m)=0$ implies $\w{a}_m\in \on{Ker}(\w{a}^\star_n)$. By the continuity of $\w{a}_n^\star$, $\on{Ker}(\w{a}^\star_n)$ is closed, so $\overline{\on{span}\{\w{a}_m: m \neq n\}}\subseteq \on{Ker}(\w{a}^\star_n)$. Since $\w{a}_n^\star(\w{a}_n)=1$, we obtain $\w{a}_n \notin \on{ker}\w{a}^\star_n$ for every $n$. If $\beta_n\neq \w{a}_n^\star(\w{x})$ for some $\w{x}\in X$ and $n\in \omega$, then, by Lemma \ref{L:wspolczynniki}, there is some $\varepsilon>0$ such that 
    \[
    \left\|\sum_{i=0}^k \w{a}^\star_i(\w{x}) \w{a}_i- \sum_{i=0}^{k'} \beta_i \w{a}_i \right\|>\varepsilon
    \]
    for every $k,k'>n$. But since $\w{x}=\sum_{n,\I} \w{a}^\star_n(\w{x}) \w{a}_n$ and $\I\neq P(\omega)$, there exists $k>n$ such that 
    \[
     \left\|\w{x}-\sum_{i=0}^k \w{a}^\star_i(\w{x}) \w{a}_i\right\|<\frac{\varepsilon}{2},
    \]
    hence 
    \[
     \left\|\w{x}-\sum_{i=0}^{k'} \beta_i \w{a}_i\right\|>\frac{\varepsilon}{2},
    \]
    for all $k'>n$, contradicting $\w{x}= \sum_{n,\J} \beta_i \w{a}_n$.
\end{proof}

Above considerations leads us to some useful Corollaries.

\begin{corollary}
Let $X$ be a Banach space and $\I$ be a nontrivial ideal. Assume that $\hat{\w{a}}=(\w{a}_n,\w{a}^\star_n)\subseteq X\times X^\star$ is an $\I$-Schauder basis of $X$ such that $\w{a}^\star_n(\w{a}_m)=\delta_{nm}$ for all $n,m\in\omega$. If $\hat{\w{a}}$ is also a $\J$-Schauder basis of $X$, for some nontrivial ideal $\J$ (in particular, for $\J=\I$), then for every $\w{x}\in X$ there exists a unique sequence of scalars $(\alpha_n)$ ($=(\w{a}_n^\star(\w{x})))$) such that $\w{x}=\sum_{n,\J}\alpha_n \w{a}_n$.
\end{corollary}

\begin{proof}
Follows from Proposition \ref{P:wspolczynniki}.
\end{proof}

\begin{corollary}
Let $X$ be a Banach space and $\I$ be a nontrivial analytic ideal. Assume that $(\w{a}_n)\subseteq X$ is such that for every $\w{x}\in X$ there exists a unique sequence of scalars $(\alpha_{n,\w{x}})$ such that $\w{x}=\sum_{n,\I}\alpha_{n,\w{x}} \w{a}_n$. If for a nontrivial ideal $\J$ and $\w{x}\in X$ there is a sequence of scalars $(\beta_n)$ such that $\w{x}=\sum_{n,\J}\beta_n \w{a}_n$, then $(\beta_n)=(\alpha_{n,\w{x}})$.
\end{corollary}

\begin{proof}
Follows from Proposition \ref{P:wspolczynniki} and \cite[Theorem A]{RKS}.
\end{proof}

Clearly the definition of Schauder basis may be considered also in a normed vector spaces, and we will do so in Section \ref{S:normed}.

\subsection{Critical ideal}

It is known (see \cite[Proof of Theorem B]{RKS}) that for each Banach space $X$ and each $\hat{\w{a}}=(\w{a}_n,\w{a}^\star_n)\subseteq X\times \IK^X$ there is an ideal $\CR(X,\hat{\w{a}})$ on $\omega$ such that the following are equivalent for every ideal $\I$ on $\omega$:
\begin{itemize}
    \item $\hat{\w{a}}$ is an $\I$-Schauder basis of $X$,
    \item $\CR(X,\hat{\w{a}})\subseteq\I$.
\end{itemize}
We say that $\hat{\w{a}}$ is an ideal Schauder basis of $X$ if $\CR(X,\hat{\w{a}})\neq\mathcal{P}(\omega)$. Obviously, if $\CR(X,\hat{\w{a}})=\Fin$, then $\hat{\w{a}}$ is a Schauder basis of $X$.

\subsection{Simple ideal Schauder bases}

In this paper we are mainly interested in ideal Schauder bases of special form:

\begin{definition}
\label{def-simple}
Let $X$ be a vector space, $\hat{\w{a}}=(\w{a}_n,\w{a}^\star_n)\subseteq X\times X^\star$ and $\hat{\w{u}}=(\w{u}_n,\w{u}^\star_n)\subseteq X\times X^\star$. We say that $\hat{\w{a}}$ is simple over $\hat{\w{u}}$ if there are $D\subseteq\omega$, an interval-to-one $h:D\to\omega$ (i.e., $h^{-1}[\{k\}]$ is either empty or an interval contained in $D$, for every $k\in\omega$) and $(\w{b}_n)_{n\in h[D]}\subseteq X\setminus\{0\}$ such that:
\[
S_n(\hat{\w{u}})(\w{x})-S_n(\hat{\w{a}})(\w{x})=
\begin{cases}
    0, & \text{if }n\in\omega\setminus D,\\
    \w{u}_{h(n)}^\star(\w{x})\w{b}_{h(n)}, & \text{if }n\in D,
\end{cases}
\]
for all $n\in \omega$ and $\w{x}\in X$. In particular, if $X$ is a Banach space and $\hat{\w{u}}$ is a $\Fin$-Schauder basis of $X$, then for a given ideal $\I$, $(S_n(\hat{\w{a}})(\w{x}))$ is $\I$-convergent to $\w{x}$ if and only if $(\w{u}_{h(n)}^\star(\w{x})\w{b}_{h(n)})_{n\in D}$ is $\I\restriction D$-convergent to zero in $X$.
\end{definition}

However the above definition may be considered technical, we decided to restrict our considerations to such ideal Schauder bases for two reasons: all our examples are of this form and this special form enabled us to classify all such ideal Schauder bases in standard Banach spaces.

Next result, in some sense, allows us to consider only a special class of Banach spaces.

\begin{theorem}
If $\hat{\w{c}}$ is simple over some normed $\Fin$-Schauder basis $\hat{\w{u}}$ (that is, $\|\w{u}_n\|=1$ for each $n$) of some Banach space $Y$ over $\mathbb{R}$, then $\CR(Y,\hat{\w{c}})=\CR(X,\hat{\w{a}})$ for some Banach space $\ell_1\subseteq X\subseteq c_0$ (equipped in coordinate-wise addition and multiplication by scalars) and some $\hat{\w{a}}$ simple over $\hat{\w{e}}$. Moreover:
\begin{itemize}
    \item[(i)] $\hat{\w{u}}$ is unconditional in $Y$ if and only if $\hat{\w{e}}$ is unconditional in $X$,
    \item[(ii)] the topology on $X$ is stronger than the topology inherited from $c_0$,
    \item[(iii)] the topology on $X$ is weaker than $\ell_1$ topology on $X$ (understood as topology generated by sets of the form $B((x_n),r)=\{(y_n)\in X : \sum_{n\in\omega} |x_n-y_n|<r\}$).
\end{itemize}
\end{theorem}

\begin{proof}
If $Y$ is an arbitrary Banach space over $\mathbb{R}$ and $\hat{\w{u}}$ is its $\Fin$-Schauder basis, then consider: 
\[
X=\left\{\w{x}\in\mathbb{R}^\omega: \sum_{n\in\omega}x_n\w{u}_n\in Y\right\}
\]
normed by: 
\[
\|\w{x}\|_X=\left\|\sum_{n\in\omega}x_n\w{u}_n\right\|_Y,
\]
for all $\w{x}\in X$. Then $X$ is a Banach space and $\ell_1\subseteq X\subseteq c_0$.

If $\hat{\w{c}}$ is simple over $\hat{\w{u}}$, then $\CR(Y,\hat{\w{c}})=\CR(X,\hat{\w{a}})$, where $\hat{\w{a}}$ is simple over $\hat{\w{e}}$ with the same witnessing $D$, $h$ and $(\w{b}_n)$ as for $\hat{\w{c}}$ and $\hat{\w{u}}$.

Item (i) follows from the fact that $X$ and $Y$ are isometric (by the isometry that moves $\hat{\w{e}}$ to $\hat{\w{u}}$). To show item (ii), observe that for any $\w{x}_n, \w{x}\in X$ and any $k\in \IN$ we have:
\begin{align*}
    |\w{u}^\star_{k+1}(\w{x}_n-\w{x})| & =  \|\w{u}^\star_{k+1}(\w{x}_n-\w{x})\w{u}_{k+1}\|_X = \|S_{k+1}(\hat{\w{u}}) (\w{x}_{n}-\w{x})- S_{k}(\hat{\w{u}}) (\w{x}_{n}-\w{x})\|_X=\cr
    & = \|(S_{k+1}-S_k)(\hat{\w{u}}) (\w{x}_{n}-\w{x})\|_X \leq 2K \|\w{x}_{n}-\w{x}\|_X,
\end{align*}
where $K$ is the basis constant of $\hat{\w{u}}$ (see \cite[Proposition 1.1.4 and Definition 1.1.5]{AK}). The final claim is clear.
\end{proof}

\section{Some lemmas}

\begin{lemma}
\label{general-lem1}
Let $\hat{\w{a}}$ be simple over a $\Fin$-Schauder basis $\hat{\w{u}}$ of a Banach space $X$, and let $D$, $h$ and $(\w{b}_n)_{n\in\omega}$ be as in Definition \ref{def-simple}. Then $\CR(X,\hat{\w{a}})$ is the ideal on $\omega$ generated by $\Fin$ and all sets of the form
    \[
    A_{\w{x}}=\left\{n\in D: |\w{u}^\star_{h(n)}(\w{x})|\geq\frac{1}{\|\w{b}_{h(n)}\|_X}\right\},
    \]
for $\w{x}\in X$.
\end{lemma}

\begin{proof}
Let $\hat{\J}(X,\hat{\w{u}},\hat{\w{a}})$ be the ideal on $\omega$ generated by $\Fin$ and all sets of the form $A_{\w{x}}$, for $\w{x}\in X$. We need to show that the following are equivalent for any ideal $\I$ on $\omega$:
\begin{itemize}
    \item[(a)] $\hat{\w{a}}$ is an $\I$-Schauder base of $X$;
    \item[(b)] $\hat{\J}(X,\hat{\w{u}},\hat{\w{a}})\subseteq\I$.
\end{itemize}

(a)$\implies$(b): Assume that $\hat{\J}(X,\hat{\w{u}},\hat{\w{a}})\not\subseteq\I$. Then there exists an $A\in\hat{\J}(X,\hat{\w{u}},\hat{\w{a}})\setminus\I$. Hence, there are $F\in\Fin$, $k\in\omega$ and $\w{x}_0,\ldots,\w{x}_k\in X$ such that $A\subseteq F\cup A_{\w{x}_0}\cup\ldots\cup A_{\w{x}_k}$. Since $A\notin\I$, we can find $j\leq k$ such that $A_{\w{x}_j}\notin\I$. Observe that $(\w{u}_{h(n)}^\star(\w{x}_j)\w{b}_{h(n)})_{n\in D}$ is not $\I\restriction D$-convergent to zero, as:
\[
\{n\in D: \|\w{u}_{h(n)}^\star(\w{x}_j)\w{b}_{h(n)}\|_X\geq 1\}=A_{\w{x}_j}\notin\I.
\]
Thus, $\hat{\w{a}}$ is not an $\I$-Schauder basis of $X$. 

(b)$\implies$(a): Assume that $\hat{\J}(X,\hat{\w{u}},\hat{\w{a}})\subseteq\I$ and fix any $\w{x}\in X$ and $\varepsilon>0$. Then $\frac{\w{x}}{\varepsilon}\in X$, so $A_{\frac{\w{x}}{\varepsilon}}\in\hat{\J}(X,\hat{\w{u}},\hat{\w{a}})\subseteq\I$. To finish the proof, note that:
\[
\{n\in D: \|\w{u}_{h(n)}^\star(\w{x})\w{b}_{h(n)}\|_X\geq \varepsilon\}=\left\{n\in D: \left|\w{u}_{h(n)}^\star\left(\frac{\w{x}}{\varepsilon}\right)\right|\geq \frac{1}{\|\w{b}_{h(n)}\|_X}\right\}=A_{\frac{\w{x}}{\varepsilon}}.
\]
\end{proof}

\begin{remark}
\label{remark-general-lem1}
Note that the above Lemma holds in all normed spaces. We will use this fact in Section \ref{S:normed}.
\end{remark}

\begin{lemma}
\label{general-lem2}
Let $\hat{\w{u}}$ be an unconditional $\Fin$-Schauder basis of some Banach space $X$. Assume that $\hat{\w{a}}$ is simple over $\hat{\w{u}}$ and $D$, $h$ and $(\w{b}_n)_{n\in\omega}$ are as in Definition \ref{def-simple}. Then:
\[
\CR(X,\hat{\w{a}})=\left\{A\subseteq\omega: A\setminus D\in\Fin\text{ and }\sum_{k\in h[A\cap D]}\frac{\w{u}_{k}}{\|\w{b}_{k}\|_X}\in X\right\}.
\]
\end{lemma}

\begin{proof}
By Lemma \ref{general-lem1}, $\CR(X,\hat{\w{a}})$ is the ideal on $\omega$ generated by $\Fin$ and all sets of the form
    \[
    A_{\w{x}}=\left\{n\in D: |\w{u}^\star_{h(n)}(\w{x})|\geq\frac{1}{\|\w{b}_{h(n)}\|_X}\right\},
    \]
for $\w{x}\in X$. Let 
\[
\J(X,\hat{\w{u}},\hat{\w{a}})=\left\{A\subseteq\omega: A\setminus D\in\Fin\text{ and }\sum_{k\in h[A\cap D]}\frac{\w{u}_{k}}{\|\w{b}_{k}\|_X}\in X\right\}.
\]

The inclusion $\J(X,\hat{\w{u}},\hat{\w{a}})\subseteq\CR(X,\hat{\w{a}})$ is obvious, as $A\in \J(X,\hat{\w{u}},\hat{\w{a}})$ means that $A\setminus D\in\Fin\subseteq\CR(X,\hat{\w{a}})$ and $\w{x}=\sum_{k\in h[A\cap D]}\frac{\w{u}_{k}}{\|\w{b}_{k}\|_X}\in X$, which gives us $A\cap D\subseteq A_{\w{x}}\in \CR(X,\hat{\w{a}})$.

Now we show $\J(X,\hat{\w{u}},\hat{\w{a}})\supseteq\CR(X,\hat{\w{a}})$. It is obvious that $\CR(X,\hat{\w{a}})\restriction\omega\setminus D=\Fin\restriction\omega\setminus D\subseteq \J(X,\hat{\w{u}},\hat{\w{a}})$. 

Let $A\in \CR(X,\hat{\w{a}})\restriction D$. Then there are $l\in\omega$ and $\w{x}_0,\ldots,\w{x}_l\in X$ such that $A\subseteq A_{\w{x}_0}\cup\ldots\cup A_{\w{x}_l}$. Define $B_0=h[A_{\w{x}_0}]$ and $B_i=h[A_{\w{x}_i}]\setminus\bigcup_{j<i}h[A_{\w{x}_j}]$ for $i=1,\ldots,l$. Note that for each $i\leq l$, since $\w{x}_i\in X$, we have:
\[
\w{y}_i=\sum_{k\in B_i\cap h[A\cap D]}\w{u}_{k}^\star(\w{x}_i)\w{u}_{k}\in X
\]
(by Lemma \ref{lem-monotonicznosc}). Hence, also $\w{z}_i=\sum_{k\in B_i\cap h[A\cap D]}\frac{\w{u}_{k}}{\|\w{b}_{k}\|}\in X$ (again, by Lemma \ref{lem-monotonicznosc}). Finally, since the sets $B_i$ are pairwise disjoint and $h[A\cap D]=\bigcup_{i\leq l}h[A\cap D]\cap B_i$, we get $\sum_{k\in h[A\cap D]}\frac{\w{u}_{k}}{\|\w{b}_{k}\|}\in X$.
\end{proof}

Next example shows that in Lemma \ref{general-lem2} the assumption about unconditionality of $\hat{\w{u}}$ cannot be dropped. 

\begin{example}
Consider $\hat{\w{u}}$ given by $\w{u}_n = \sum_{i=0}^n \w{e}_n $ and $\w{u}_n^\star=\w{e}_n^\star - \w{e}_{n+1}^\star$, for all $n\in\omega$, which is a non-unconditional $\Fin$-Schauder basis of $c_0$ (see \cite[Example 3.1.2]{AK}).

Let $\hat{\w{a}}$ be given by $\w{a}_n = \sum_{i=0}^n \w{u}_n $ and $\w{a}_n^\star=\w{u}_n^\star - \w{u}_{n+1}^\star$, for all $n\in\omega$.

Observe that:
\[
S_n(\hat{\w{a}})(\w{x})=\sum_{i=0}^n (\w{u}_i^\star-\w{u}_{i+1}^\star) (\w{x})\sum_{k=0}^i \w{u}_k = \sum_{i=0}^n \w{u}_i^\star (\w{x}) \w{u}_i - \w{u}_{n+1}^\star(\w{x}) \sum_{i=0}^n \w{u}_i=S_n(\hat{\w{u}})(\w{x})-\w{u}_{n+1}^\star(\w{x}) \sum_{i=0}^n \w{u}_i,
\]
for any $\w{x}\in c_0$. Thus, 
\[
S_n(\hat{\w{u}})(\w{x})-S_n(\hat{\w{a}})(\w{x})=\w{u}_{n+1}^\star(\w{x}) \sum_{i=0}^n \w{u}_i
\]
and $\hat{\w{a}}$ is simple over $\hat{\w{u}}$ (as witnessed by $D=\omega$, $h(n)=n+1$ and $\w{b}_{n+1}=\sum_{i=0}^n \w{u}_i$).

Note that: 
\[
\omega\notin\left\{A\subseteq\omega: A\setminus D\in\Fin\text{ and }\sum_{k\in h[A\cap D]}\frac{\w{u}_{k}}{\|\w{b}_{k}\|_{c_0}}\in c_0\right\}.
\]
Indeed, $\w{x}=\sum_{k\in h[\omega]}\frac{\w{u}_{k}}{\|\w{b}_k\|_{c_0}}=\sum_{n\in\omega}\frac{\w{u}_{n+1}}{n+1}\notin c_0$, since $\w{e}_0(\w{x})=\sum_{n\in\omega}\frac{1}{n+1}=\infty$.

We will show that $\omega\in\CR(c_0,\hat{\w{a}})$. Consider $\w{x}=\sum_{n\in\omega}\frac{\w{e}_{2n}}{n+1}\in c_0$ and $\w{y}=\sum_{n\in\omega}\frac{\w{e}_{2n+1}}{n+1}\in c_0$. We have:
\[
A_{\w{x}}=\left\{n\in\omega: |\w{u}^\star_{n+1}(\w{x})|\geq\frac{1}{\|\w{b}_{n+1}\|_{c_0}}\right\}=\left\{n\in\omega: |\w{e}_{n+1}^\star(\w{x}) - \w{e}_{n+2}^\star(\w{x})|\geq\frac{1}{n+1}\right\}.
\]
Note that if $n\in\{2k+1:k\in\omega\}$, i.e., $n=2m+1$ for some $m\in\omega$, then 
\[
|\w{e}_{n+1}^\star(\w{x}) - \w{e}_{n+2}^\star(\w{x})|=|\w{e}_{2m+2}^\star(\w{x}) - \w{e}_{2m+3}^\star(\w{x})|=\w{e}_{2m+2}^\star(\w{x})=\frac{1}{m+2}\geq \frac{1}{n+1}.
\]
This shows that $\{2k+1:k\in\omega\}\subseteq A_{\w{x}}$.

Similarly one can show that $\{2k:k\in\omega\}\subseteq A_{\w{y}}$. Hence, applying Lemma \ref{general-lem1}, $\omega=\{2k:k\in\omega\}\cup\{2k+1:k\in\omega\}\in\CR(c_0,\hat{\w{a}})$.
\end{example}

\begin{proposition}
\label{general-P-ideal}
If $\hat{\w{u}}$ is an unconditional $\Fin$-Schauder basis of some Banach space $X$ and $\hat{\w{a}}$ is simple over $\hat{\w{u}}$, then $\CR(X,\hat{\w{a}})$ is a P-ideal.
\end{proposition}

\begin{proof}
From Lemma \ref{general-lem2} we know that:
\[
\CR(X,\hat{\w{a}})=\left\{A\subseteq\omega: A\setminus D\in\Fin\text{ and }\sum_{k\in h[A\cap D]}\frac{\w{u}_{k}}{\|\w{b}_{k}\|_X}\in X\right\},
\]
where $D$, $h$ and $(\w{b}_n)_{n\in\omega}$ are as in Definition \ref{def-simple}. 

Let $(A_k)\subseteq \CR(X,\hat{\w{a}})$. For each $k\in\omega$, since $\sum_{i\in h[A_k\cap D]}\frac{\w{u}_{i}}{\|\w{b}_{i}\|_X}\in X$ and $h$ is interval-to-one, there is $m_k\in\omega$ such that $\left\|\sum_{i\in h[(A_k\cap D)\setminus m_k]}\frac{\w{u}_{i}}{\|\w{b}_{i}\|_X}\right\|_X<\frac{1}{2^k}$. Define $A=\bigcup_{k\in\omega}((A_k\cap D)\setminus m_k)$. Then $A\subseteq D$ and $A_k\setminus A\in\Fin$, for each $k$. We need to show that $A\in\CR(X,\hat{\w{a}})$, i.e., $\sum_{i\in h[A]}\frac{\w{u}_{i}}{\|\w{b}_{i}\|_X}\in X$. 

Given any $\varepsilon>0$, find $j\in\omega\setminus\{0\}$ such that $\frac{1}{2^{j-1}}<\varepsilon$ and $m\in\omega$ such that 
\[
\sum_{k\leq j}\left\|\sum_{i\in h[A_k\cap D]\setminus m}\frac{\w{u}_{i}}{\|\w{b}_{i}\|_X}\right\|_X<\frac{1}{2^{j}}.
\]
Observe that
\begin{align*}
    \left\|\sum_{i\in h[A]\setminus m}\frac{\w{u}_{i}}{\|\w{b}_{i}\|_X}\right\|_X & \leq\sum_{k\leq j}\left\|\sum_{i\in h[A\cap A_k]\setminus m}\frac{\w{u}_{i}}{\|\w{b}_{i}\|_X}\right\|_X+\sum_{k>j}\left\|\sum_{i\in h[A\cap A_k]}\frac{\w{u}_{i}}{\|\w{b}_{i}\|_X}\right\|_X< \cr
    & <\frac{1}{2^{j}}+\sum_{k>j}\frac{1}{2^k}=\frac{1}{2^{j-1}}<\varepsilon.
\end{align*}
\end{proof}

\section{Banach spaces with a Schauder basis}

\begin{theorem}
\label{thm-distinguishing}
Let $X$ be any Banach space with a $\Fin$-Schauder basis $\hat{\w{u}}$. For any infinite $E\subseteq\omega$ there is $\hat{\w{a}}\subseteq X\times X^\star$ simple over $\hat{\w{u}}$ such that the following are equivalent for any ideal $\I$:
\begin{itemize}
    \item[(a)] $\hat{\w{a}}$ is an $\I$-Schauder basis of $X$;
    \item[(b)] $E\in\I$;
    \item[(c)] $\I_E\subseteq\I$, where $\I_E$ is the ideal generated by $\Fin\cup\{E\}$.
\end{itemize}
In particular, $\CR(X,\hat{\w{a}})=\I_E$.
\end{theorem}

\begin{proof}
Fix a $\Fin$-Schauder basis $\hat{\w{u}}$ of $X$. Without loss of generality we may assume that $\hat{\w{u}}$ is normed, that is, $\|\w{u}_n\|_X=1$ for all $n\in\omega$. Let $(I_k)$ be a sequence of pairwise disjoint intervals in $\omega$ such that $E=\bigcup_k I_k$ and $\max I_k+1<\min I_{k+1}$. Denote $b_k=\min I_k$ and let 
\[
d_k=\min\left\{\left\|\sum_{i=j}^{\max I_k+1} \w{u}_{i}\right\|_X: b_k<j\leq \max I_k+1\right\},
\]
for all $k\in\omega$.

Define 
\[
    \w{a}_n =
    \begin{cases*}
    \w{u}_{n}+\frac{2^n}{d_k}\cdot\sum_{i=n+1}^{\max I_k+1} \w{u}_{i}     & if $n=b_k$ for some $k$, \\
    \w{u}_n & otherwise,  \\
    \end{cases*}
  \]
and
\[
    \w{a}_n^\star =
    \begin{cases*}
    \w{u}_n^\star & if  $n-1\notin E$  \\
    \w{u}_{n}^\star - \frac{2^{b_k}}{d_k}\cdot \w{u}_{b_k}^\star & if $n-1\in I_k$ for some $k$.
    \end{cases*}
  \]

Clearly, $\w{a}^\star_n(\w{a}_m)=\delta_{nm}$ for all $n,m\in\omega$.

The equivalence of items (b) and (c) is obvious. We will show that items (a) and (b) are equivalent.

Observe that:
\[
S_n(\hat{\w{u}})(\w{x})-S_n(\hat{\w{a}})(\w{x})=\begin{cases*}
    0, & if  $n\notin E$  \\
    \w{u}^\star_{b_k}(\w{x})\frac{2^{b_k}}{d_k}\sum_{i=j}^{\max I_k+1}\w{u}_i, & if $n=b_k+j$ for some $k\in\omega$ and $0\leq j<|I_k|$.
    \end{cases*}
\]
Hence, if $E\in\I$, then $\hat{\w{a}}$ is an $\I$-Schauder basis of $X$. 

We need to show that if $E\notin\I$, then $\hat{\w{a}}$ is not an $\I$-Schauder basis of $X$. Define $\w{x}=\sum_{n\in\omega}\frac{1}{2^n}\w{u}_n$ (note that $\w{x}\in X$, since $\|\sum_{n>k}\frac{1}{2^n}\w{u}_n\|_X\leq\frac{1}{2^k}$ for all $k\in\omega$). Then 
\[
\{n\in\omega: \|S_n(\hat{\w{u}})(\w{x})-S_n(\hat{\w{a}})(\w{x})\|_X\geq 1\}\supseteq E\notin\I,
\]
since given any $n\in E$, find $k\in\omega$ and $0\leq j<|I_k|$ such that $n=b_k+j\in I_k$, and observe that:
\[
\|S_n(\hat{\w{u}})(\w{x})-S_n(\hat{\w{a}})(\w{x})\|_X=\left\|\w{u}^\star_{b_k}(\w{x})\frac{2^{b_k}}{d_k}\sum_{i=j}^{\max I_k+1}\w{u}_i\right\|_X\geq |2^{b_k}\w{u}^\star_{b_k}(\w{x})|=1.
\]
\end{proof}

\begin{corollary}
Let $X$ be any Banach space with a $\Fin$-Schauder basis. The following are equivalent for any ideals $\I$ and $\J$:
\begin{itemize}
    \item[(a)] There is an $\I$-Schauder basis of $X$ which is not a $\J$-Schauder basis of $X$;
    \item[(b)] $\I\not\subseteq\J$.
\end{itemize}
In particular, if $\I\neq\Fin$, then there is an $\I$-Schauder basis, which is not a $\Fin$-Schauder basis.
\end{corollary}

\begin{proof}
It follows from Theorem \ref{thm-distinguishing}.
\end{proof}

\begin{corollary}
If $X$ is a Banach space with a $\Fin$-Schauder basis $\hat{\w{u}}$, then there is a family $\mathcal{F}\subseteq(X\times X^\star)^\omega$ of cardinality $2^\omega$ such that each $\hat{\w{a}}\in\mathcal{F}$ is simple over $\hat{\w{u}}$ and $\CR(X,\hat{\w{a}})\neq\CR(X,\hat{\w{b}})$, for all distinct $\hat{\w{a}},\hat{\w{b}}\in\mathcal{F}$. 
\end{corollary}

\begin{proof}
Let $\mathcal{A}\subseteq\mathcal{P}(\omega)$ be an almost disjoint family (i.e., $A\cap B\in\Fin$ for all distinct $A,B\in\mathcal{A}$) of cardinality $2^\omega$. For each $A\in\mathcal{A}$, apply Theorem \ref{thm-distinguishing} to get $\hat{\w{a}}_A\in(X\times X^\star)^\omega$ such that $\CR(X,\hat{\w{a}}_A)=\I_A$. Let $\mathcal{F}=\{\hat{\w{a}}_A: A\in\mathcal{A}\}$ and observe that $\I_A\neq\I_B$ whenever $A,B\in\mathcal{A}$ are distinct.
\end{proof}

\begin{remark}
Note that Theorem \ref{thm-distinguishing} implies that there is no single ideal $\I$ suitable for all ideal Schauder bases in such a way that for every $\hat{\w{a}}$ and $\J$, if $\hat{\w{a}}$ is an $\J$-Schauder basis, then it is also an $\I$-Schauder basis.
\end{remark}

\section{Spaces \texorpdfstring{$\ell_p$}{l }}

\subsection{Reexaming the already known example}

We start by investigating the example provided in \cite[Example 1]{CGK}.

\begin{proposition}
\label{prop-known-example}
Let $\w{a}_n = \sum_{i=0}^n \w{e}_n $ and $\w{a}_n^\star=\w{e}_n^\star - \w{e}_{n+1}^\star$ for all $n\in\omega$. Then $\hat{\w{a}}=(\w{a}_n,\w{a}^\star_n)$ is simple over $\hat{\w{e}}$ and $\CR(\ell_p,\hat{\w{a}})=\I_{1/n}$, for every $1\leq p<\infty$.
\end{proposition}

\begin{proof}
Observe that:
\[
S_n(\hat{\w{a}})(\w{x})=\sum_{i=0}^n (\w{e}_i^\star-\w{e}_{i+1}^\star) (\w{x})\sum_{k=0}^i \w{e}_k = \sum_{i=0}^n \w{e}_i^\star (\w{x}) \w{e}_i - \w{e}_{n+1}^\star(\w{x}) \sum_{i=0}^n \w{e}_i=S_n(\hat{\w{e}})(\w{x})-\w{e}_{n+1}^\star(\w{x}) \sum_{i=0}^n \w{e}_i,
\]
for any $\w{x}\in\mathbb{R}^\omega$. Thus:
\[
S_n(\hat{\w{e}})(\w{x})-S_n(\hat{\w{a}})(\w{x})=\w{e}_{n+1}^\star(\w{x}) \sum_{i=0}^n \w{e}_i
\]
and $\hat{\w{a}}$ is simple over $\hat{\w{e}}$ (as witnessed by $D=\omega$, $h(n)=n+1$ and $\w{b}_{n+1}=\sum_{i=0}^n \w{e}_i$). Hence, applying Lemma \ref{general-lem2}, we have:
\begin{align*}
\CR(\ell_p,\hat{\w{a}}) & =\left\{A\subseteq\omega: \sum_{n\in A}\frac{\w{e}_{n+1}}{\|\sum_{i=0}^n \w{e}_i\|_{\ell_p}}\in \ell_p\right\}=\left\{A\subseteq\omega: \sum_{n\in A}\frac{\w{e}_{n+1}}{(n+1)^{1/p}}\in \ell_p\right\}=\cr
& =\left\{A\subseteq\omega: \sum_{n\in A}\frac{1}{n+1}<\infty\right\}=\I_{1/n},  
\end{align*}
for every $1\leq p<\infty$.
\end{proof}

\begin{remark}
The sequence $\hat{\w{a}}$ from \cite[Example 1]{CGK} and Proposition \ref{prop-known-example} is a non-unconditional $\Fin$-Schauder basis of $c_0$ (see \cite[Example 3.1.2]{AK}).
\end{remark}

\subsection{Single basis with different critical ideals}

In Proposition \ref{prop-known-example}, we observed a single $\hat{\w{a}}$ being an ideal Schauder base of each $\ell_p$, for $1\leq p<\infty$, such that the critical ideal was the same for each $1\leq p<\infty$. In this Subsection we will construct $\hat{\w{a}}$ being an ideal Schauder base of every $\ell_p$, $1\leq p<\infty$, but with different critical ideals.

\begin{proposition}
\label{prop-different_critical}
There are a partition $(A_k)$ of $\omega$ into infinite sets and $\hat{\w{a}}$ such that 
\[
\CR(\ell_p,\hat{\w{a}})=\left\{A\subseteq\omega:A\cap A_0\in\Fin\text{ and }\sum_{k\in\omega\setminus\{0\}}\frac{|A\cap A_k|}{k^p}<\infty\right\},
\]
for each $1\leq p<\infty$. Moreover, if $1\leq p<q<\infty$, then $\CR(\ell_q,\hat{\w{a}})\not\subseteq\CR(\ell_p,\hat{\w{a}})$. In particular, the ideals $\CR(\ell_p,\hat{\w{a}})$ are summable and pairwise distinct.
\end{proposition}

\begin{proof}
Let $(A_k)$ be any partition of $\omega$ into infinite sets such that $A_0=\{2n:n\in\omega\}$. 

For each $n\in\omega$ find $m_n\in\omega$ such that $n\in A_{m_n}$ and define:
\[
\w{a}_n=\w{e}_{n}+m_n\cdot \w{e}_{n+1}\text{ and } \w{a}_n^\star=\w{e}_{n}^\star - m_{n-1}\cdot \w{e}_{n-1}^\star
\]
(here we put $m_{-1}=0$). Then 
\begin{align*}
    \w{a}_n^\star(\w{a}_k) & = \w{e}_{n}^\star(\w{e}_{k})+m_k\w{e}_{n}^\star(\w{e}_{k+1}) - m_{n-1}\w{e}_{n-1}^\star(\w{e}_{k})- m_{n-1}m_k\w{e}_{n-1}^\star(\w{e}_{k+1})=\cr
    & =\delta_{nk}+m_k\delta_{n(k+1)}-m_{n-1} \delta_{(n-1)k}-m_{n-1}m_k\delta_{(n-1)(k+1)}=\delta_{nk},
\end{align*}
since if $n-1=k+1$, then one of $n-1,k\in A_0$ hence one of $m_{n-1},m_k$ equals $0$. 

Note that for any $\w{x} \in \ell_p$ we have:
\[
S_n(\hat{\w{e}})(\w{x})-S_n(\hat{\w{a}})(\w{x})=\w{e}_{n}^\star(\w{x})m_n\w{e}_{n+1}.
\]
In particular, $\hat{\w{a}}$ is simple over $\hat{\w{e}}$ (as witnessed by $D=\{n\in\omega:m_n\neq 0\}=\omega\setminus A_0=\{2n+1:n\in\omega\}$, $h(n)=n$ and $\w{b}_n=m_n\w{e}_{n+1}$). Hence, using Lemma \ref{general-lem2}, we have:
\begin{align*}
\CR(\ell_p,\hat{\w{a}}) & =\left\{A\subseteq\omega: A\setminus D\in\Fin\text{ and }\sum_{n\in A\cap D}\frac{\w{e}_{n}}{\left\|m_n\w{e}_{n+1}\right\|_{\ell_p}}\in \ell_p\right\}=\cr
& =\left\{A\subseteq\omega: A\setminus D\in\Fin\text{ and }\sum_{n\in A\cap D}\frac{\w{e}_{n}}{m_n}\in \ell_p\right\}=\cr
& =\left\{A\subseteq\omega: A\setminus D\in\Fin\text{ and }\sum_{n\in A\cap D}\frac{1}{m_n^p}<\infty\right\}=\cr
& =\left\{A\subseteq\omega:A\cap A_0\in\Fin\text{ and }\sum_{k\in\omega\setminus\{0\}}\frac{|A\cap A_k|}{k^p}<\infty\right\}. 
\end{align*}

The ideal $\CR(\ell_p,\hat{\w{a}})$ is summable, since it is equal to $\I_{\w{r}}$ for $\w{r}=(r_n)$ given by:
\[
    r_n =
    \begin{cases*}
    1    & if $n\in A_0$,\\
    \frac{1}{k^p} & if $n\in A_k$ for some $k>0$.
    \end{cases*}
  \]

To finish the proof, we will show that $1\leq p<q<\infty$ implies $\CR(\ell_q,\hat{\w{a}})\not\subseteq\CR(\ell_p,\hat{\w{a}})$. Indeed, any set $A\subseteq\omega$ such that $A\cap A_0=\emptyset$ and $|A\cap A_k|=\lceil k^p\rceil$, for $k\in\omega\setminus\{0\}$, is in $\CR(\ell_q,\hat{\w{a}})$, but not in $\CR(\ell_p,\hat{\w{a}})$.
\end{proof}

\subsection{Ideal Schauder basis of \texorpdfstring{$\ell_p$}{l } for each summable ideal}

\begin{definition}
If $\w{t}\in\omega^\omega$ is increasing with $t_0=0$ and $\w{r}$ is a sequence of positive reals such that $\sum_{n\in\omega}r_n=\infty$, then we define the ideal:
\[\I_{\w{r}}^\w{t}=\{A\subseteq\omega: \{k\in\omega: A\cap [t_k,t_{k+1})\neq\emptyset\}\in\I_{\w{r}}\}.
\]
\end{definition}

We will show that each ideal of the form $\I_{\w{r}}^\w{t}$ is critical in $\ell_p$, where $1\leq p<\infty$, for some $\hat{\w{a}}$ simple over $\hat{\w{e}}$. We will need the following technical lemma, which holds in a wider class of normed spaces. Recall that even though we are interested in ideal Schauder bases, Definition \ref{def-simple} states the property of a pair $(\hat{\w{a}},\hat{\w{u}})$ without requiring that any of $\hat{\w{a}},\hat{\w{u}}$ is a basis in some sense. 

\begin{lemma}
\label{lem:any_sequence}
Let $\w{t}\in\omega^\omega$ be increasing with $t_0=0$ and $\w{s}$ be a sequence of positive reals. Let also $X$ be any normed space with a $\Fin$-Schauder basis $\hat{\w{u}}$. Then there is $\hat{\w{a}}$ simple over $\hat{\w{e}}$ with the witnesses $D=\omega$, $h:\omega\to\omega$ given by $h\restriction[t_k,t_{k+1})=t_{k+1}$ and some sequence $(\w{b}_{t_{k+1}})_{k\in\omega}\subseteq X\setminus\{0\}$ such that $\|\w{b}_{t_{k+1}}\|_X=s_k$. Moreover, each $\w{a}_n$ is a finite linear combination of the vectors $\w{u}_k$, for $k\in\omega$.
\end{lemma}

\begin{proof}
Define inductively the sequence $(c_k)\subseteq\mathbb{R}$ by $c_0=1$ and 
\[
c_{k+1}=\frac{1}{s_k} \cdot \left\|\sum_{i=0}^k c_i\w{u}_{t_i}\right\|_X
\]
for all $k\in\omega$. Set:
\[
    \w{a}_n =
    \begin{cases*}
    \sum_{i=0}^{k}   c_i\w{u}_{t_i},     & if $n=t_k$ for some $k\in \omega$,\\
    \w{u}_n, & otherwise,  \\
    \end{cases*}
  \]
and
\[
    \w{a}_n^\star =
    \begin{cases*}
    \frac{\w{u}_{t_k}^\star}{c_k} - \frac{u_{t_{k+1}}^\star}{c_{k+1}}, & if $n=t_k$ for some $k\in \omega$,\\
    \w{u}_n^\star, & otherwise.
    \end{cases*}
\]
Clearly, $\w{a}_n^\star(\w{a}_m)=\delta_{nm}$ for all $n,m\in\omega$. Indeed, it is obvious whenever $n \notin \{t_k:k\in\omega\}$ or $m \notin \{t_k:k\in\omega\}$, and the remaining follows from the case $t_k=k$, which is easy. 

Denote $T=\{t_k:k\in\omega\}$ and let $g:\omega\to\omega$ be given by:
\[
g(n)=k\ \Longleftrightarrow n\in[t_k,t_{k+1}).
\]
Note that $h(n)=t_{g(n)+1}$ for all $n\in\omega$ (where the function $h$ is defined in the formulation of this lemma). Observe that for any $\w{x} \in X$ we have:
\begin{align*}
S_n(\hat{\w{a}})(\w{x}) & = \sum_{i \in (n+1)\setminus T} \w{u}_i^\star (x) \w{u}_i + \sum_{j=0}^{g(n)} \left(\frac{\w{u}_{t_j}^\star}{c_j} - \frac{u_{t_{j+1}}^\star}{c_{j+1}}\right)(\w{x})\sum_{i=0}^{j} c_i\w{u}_{t_i}
= \cr
& =S_n(\hat{\w{u}})(\w{x})-\w{u}_{t_{g(n)+1}}^\star(\w{x})\frac{\sum^{g(n)}_{i=0} c_i \w{u}_{t_i}}{c_{g(n)+1}}.
\end{align*} 
Thus, $\hat{\w{a}}$ is simple over $\hat{\w{e}}$ as witnessed by $D=\omega$, $h$ and $(\w{b}_{t_{k+1}})_{k\in\omega}\subseteq X\setminus\{0\}$ such that:
\[
\w{b}_{t_{k+1}}=\frac{\sum^{k}_{i=0} c_i \w{e}_{t_i}}{c_{k+1}},
\]
for all $k\in\omega$. In particular, $\|\w{b}_{t_{k+1}}\|_X=s_k$, for all $k\in\omega$.
\end{proof}

\begin{proposition}
\label{prop-summable}
Let $\w{t}\in\omega^\omega$ be increasing with $t_0=0$ and $\w{r}$ be a sequence of positive reals such that $\sum_{n\in\omega}r_n=\infty$. Then for each $1\leq p<\infty$ there is $\hat{\w{a}}$ simple over $\hat{\w{e}}$ such that $\CR(\ell_p,\hat{\w{a}})=\I_{\w{r}}^\w{t}$.
\end{proposition}

\begin{proof}
Fix $1\leq p<\infty$ and apply Lemma \ref{lem:any_sequence} for $\w{t}\in\omega^\omega$ and $\w{s}$ given by $s_k=r_k^{1/p}$, for all $k\in\omega$, to get $\hat{\w{a}}$ simple over $\hat{\w{e}}$ with the witnesses $D=\omega$, $h:\omega\to\omega$ given by $h\restriction[t_k,t_{k+1})=t_{k+1}$ and some sequence $(\w{b}_{t_{k+1}})_{k\in\omega}$ such that $\|\w{b}_{t_{k+1}}\|_{\ell_p}=r_k^{1/p}$.

Applying Lemma \ref{general-lem2} and denoting:
\[
\tilde{A}=\left\{k\in\omega:A\cap [t_k,t_{k+1})\neq\emptyset\right\}
\]
for all $A\subseteq\omega$, we have:
\begin{align*}
\CR(\ell_p,\hat{\w{a}}) & =\left\{A\subseteq\omega: \sum_{k\in h[A]}\frac{\w{e}_{k}}{\left\|\w{b}_k\right\|_{\ell_p}}\in \ell_p\right\}=\cr
& =\left\{A\subseteq\omega: \sum_{k\in \tilde{A}}\w{e}_{t_{k+1}}r_k^{1/p}\in \ell_p\right\}=\cr
& =\left\{A\subseteq\omega: \sum_{k\in \tilde{A}}r_k<\infty\right\}=\I^{\w{t}}_{\w{r}}.    
\end{align*}
\end{proof}

\begin{corollary}
For each $1\leq p<\infty$, there are $2^\omega$ ideal Schauder bases of $\ell_p$ with pairwise non-isomorphic critical ideals.
\end{corollary}

\begin{proof}
It follows from Proposition \ref{prop-summable} and \cite[Theorem 1]{MezaGuzman}.
\end{proof}

Thus far, each presented ideal Schauder basis of $\ell_p$ (i.e., the ones from Theorem \ref{thm-distinguishing} and Propositions \ref{prop-known-example} and \ref{prop-different_critical}) had a summable critical ideal. Now, using Proposition \ref{prop-summable}, we will show an ideal Schauder basis of $\ell_p$ with a non-summable critical ideal. 

\begin{corollary}
\label{cor:non-summable}
For any $1\leq p<\infty$ there is $\hat{\w{a}}$ such that $\CR(\ell_p,\hat{\w{a}})$ is a non-summable $\bf{F_\sigma}$ ideal. 
\end{corollary}

\begin{proof}
Let $t_0=0$ and $t_k=2^k$ for all $k\in\omega\setminus\{0\}$. Denote $\w{t}=(t_k)$ and $I_k=[t_k,t_{k+1})$, for all $k\in\omega$. In particular, $|I_0|=2$ and $|I_k|=2^k$ for all $k>0$. By Proposition \ref{prop-summable}, for each $1\leq p<\infty$ there is $\hat{\w{a}}$ such that $\CR(\ell_p,\hat{\w{a}})=\I_{1/n}^\w{t}$. Therefore, we need to demonstrate that $\I_{1/n}^\w{t}$ is non-summable and $\bf{F_\sigma}$.

To verify that $\I_{1/n}^\w{t}$ is $\bf{F_\sigma}$, note that $\I_{1/n}^\w{t}=\phi^{-1}[\I_{1/n}]$, where $\phi:\mathcal{P}(\omega)\to \mathcal{P}(\omega)$ given by $\phi(A)=\{k\in\omega: A\cap I_k\neq\emptyset\}$, for all $A\subseteq\omega$, is continuous.

Suppose towards contradiction that $\I_{1/n}^\w{t}$ is a summable ideal, i.e., that there is $\w{r}=(r_n)$ such that $r_n>0$, $\sum_{n\in\omega}r_n=\infty$ and $\I_{1/n}^\w{t}=\I_{\w{r}}$. Denote $d_k=\min\{r_n: n\in I_k\}$ and $n_k=\min\{n\in I_k: r_n=d_k\}$. Observe that $Z=\{k\in\omega: 2^k\cdot d_k\geq 1\}$ is infinite, since otherwise we would get:
\[
\sum_{k\in\omega}r_{n_k}=\sum_{k\in Z}d_k+\sum_{k\in\omega\setminus Z}d_k<\sum_{k\in Z}d_k+\sum_{k\in\omega\setminus Z}\frac{1}{2^k}<\infty,
\]
which shows that $\{n_k:k\in\omega\}\in\I_{\w{r}}$ and contradicts $\I_{1/n}^\w{t}=\I_{\w{r}}$, as $\{n_k:k\in\omega\}\notin\I_{1/n}^\w{t}$. However, as $Z$ is infinite, we can find an infinite $Z'\subset Z$ such that $Z'\in\I_{1/n}$. Then $\bigcup_{k\in Z'}I_k\in \I_{1/n}^\w{t}$, but:
\[
\sum_{n\in \bigcup_{k\in Z'}I_k}r_n\geq \sum_{k\in Z'}|I_k|d_k\geq\sum_{k\in Z'}2^k\cdot d_k\geq \sum_{k\in Z'}1=\infty,
\]
so $\bigcup_{k\in Z'}I_k\notin \I_{\w{r}}$, which again contradicts $\I_{1/n}^\w{t}=\I_{\w{r}}$.
\end{proof}

\subsection{Characterization}

\begin{theorem}
\label{char:lp}
Let $1\leq p<\infty$. The following are equivalent for any nontrivial ideal $\I$:
\begin{itemize}
    \item[(a)] there is $\hat{\w{a}}$ simple over $\hat{\w{e}}$ such that $\CR(\ell_p,\hat{\w{a}})=\I$,
    \item[(b)] $\I=\I^\w{t}_\w{r}$, for some increasing $\w{t}\in\omega^\omega$ with $t_0=0$ and some sequence $\w{r}$ of positive reals such that $\sum_{n\in\omega}r_n=\infty$.
\end{itemize}
\end{theorem}

\begin{proof}
(b)$\implies$(a): This is Proposition \ref{prop-summable}.

(a)$\implies$(b): Fix some $\hat{\w{a}}$ simple over $\hat{\w{e}}$ and let $D$, $h$ and $(\w{b}_n)$ be as in Definition \ref{def-simple}. Find a partition $(I_n)$ of $\omega$ into consecutive intervals such that either $I_n=h^{-1}[\{k\}]$ for some $k\in\omega$ (note that in this case $I_n\subseteq D$) or $I_n=\{i\}$ for some $i\in\omega\setminus D$. Denote $S=\{n\in\omega: I_n=h^{-1}[\{k\}]\text{ for some }k\in\omega\}$. 

For each $n\in\omega$ put $t_n=\min I_n$ and
\[
r_n=\begin{cases}
    \left(\frac{1}{\|\w{b}_n\|_{\ell_p}}\right)^p, & \text{if } n\in S,\\
    1, & \text{otherwise}.
\end{cases}
\]
Then, using Lemma \ref{general-lem2}, we get that $\I^\w{t}_\w{r}=\CR(\ell_p,\hat{\w{a}})$. Since $\CR(\ell_p,\hat{\w{a}})=\I$ is an ideal, $\omega\notin\CR(\ell_p,\hat{\w{a}})$, which implies that $\sum_{n\in\omega}r_n=\infty$.
\end{proof}

\section{The space \texorpdfstring{$c_0$}{c } and the James' space}\label{S:c0}

\subsection{A non-\texorpdfstring{$\bf{F_\sigma}$}{} critical ideal}
Next result gives us first example of a non-$\bf{F_\sigma}$ ideal of the form $\CR(X,\hat{\w{a}})$.

\begin{proposition}
\label{0xFin}
There is $\hat{\w{a}}$ simple over $\hat{\w{e}}$ such that $\CR(c_0,\hat{\w{a}})\cong\{\emptyset\}\otimes\Fin$.
\end{proposition}

\begin{proof}
Let $(A_k)$ and $\hat{\w{a}}$ be as in the proof of Proposition \ref{prop-different_critical}, i.e., $(A_k)$ is a partition of $\omega$ into infinite sets such that $A_0=\{2n:n\in\omega\}$, $\w{a}_n=\w{e}_{n}+m_n\cdot \w{e}_{n+1}$ and $\w{a}_n^\star=\w{e}_{n}^\star - m_{n-1}\cdot \w{e}_{n-1}^\star$, where $m_n$ is given by $n\in A_{m_n}$.

Define:
\[
\J=\left\{A\subset\omega:\forall_{k\in\omega}A\cap A_k\in\Fin\right\}. 
\]
Clearly, $\J$ is isomorphic with $\{\emptyset\}\otimes\Fin$, as witnessed by any bijection $f:\omega^2\to\omega$ such that $f[\{k\}\times\omega]=A_k$, for all $k\in\omega$.

We will show that $\CR(c_0,\hat{\w{a}})=\J$. Similarly as in the proof of Proposition \ref{prop-different_critical}, for any $\w{x} \in c_0$ we have: 
\[
S_n(\hat{\w{e}})(\w{x})-S_n(\hat{\w{a}})(\w{x})=\w{e}_{n}^\star(\w{x})m_n\w{e}_{n+1},
\]
i.e., $\hat{\w{a}}$ is simple over $\hat{\w{e}}$ as witnessed by $D=\{n\in\omega:m_n\neq 0\}=\omega\setminus A_0=\{2n+1:n\in\omega\}$, $h(n)=n$ and $\w{b}_n=m_n\w{e}_{n+1}$. Hence, applying Lemma \ref{general-lem2}, we have:
\begin{align*}
\CR(c_0,\hat{\w{a}}) & =\left\{A\subseteq\omega: A\setminus D\in\Fin\text{ and }\sum_{n\in A\cap D}\frac{\w{e}_{n}}{\left\|m_n\w{e}_{n+1}\right\|_{c_0}}\in c_0\right\}=\cr
& =\left\{A\subseteq\omega: A\setminus D\in\Fin\text{ and }\sum_{n\in A\cap D}\frac{\w{e}_{n}}{m_n}\in c_0\right\}=\cr
& =\left\{A\subseteq\omega: A\setminus D\in\Fin\text{ and }\lim_{n\in A\cap D}\frac{1}{m_n}=0\right\}=\J.   
\end{align*}
\end{proof}

\begin{remark}
Note that $\hat{\w{a}}$ from Proposition \ref{0xFin} yields an interesting example of a basis with different behaviour in various spaces: 
\[\CR(c_0,\hat{\w{a}})\cong\{\emptyset\}\otimes\Fin,\]
but for $p\in [1,\infty)$ we have:
\[
\CR(\ell_p,\hat{\w{a}})=\left\{A\subseteq\omega:A\cap A_0\in\Fin\text{ and }\sum_{k\in\omega\setminus\{0\}}\frac{|A\cap A_k|}{k^p}<\infty\right\}
\]
(see Proposition \ref{prop-different_critical}).
\end{remark}

\subsection{James' space}

We recall the definition of the James' space. Let $\mathfrak{J}$ consist of all $\w{x}\in c_0$ such that:
\[
\|x\|_{\mathfrak{J}}=\sup\left\{\left(\sum_{i=1}^n \left(x_{p_i}-x_{p_{i-1}}\right)^2\right)^{1/2}:n\in\omega\text{ and }0\leq p_0<p_1<\ldots<p_n\right\}<+\infty.
\]
Then the above formula defines a norm on $\mathfrak{J}$, making it a Banach space. This space is called the James' space. It is known that $\hat{\w{e}}$ is a non-unconditional Schauder basis of $\mathfrak{J}$. In fact, every Schauder basis of $\mathfrak{J}$ is not unconditional. Thus, in our examinations of the James' space, we cannot use Lemma \ref{general-lem2} and have to rely on Lemma \ref{general-lem1} -- this property of $\mathfrak{J}$ makes it particularly interesting in our context. For more on James' space see \cite{AK}.

The following observation will be crucial in our considerations of the James' space.

\begin{lemma}
\label{lem:James}
Let $\hat{\w{a}}$ be simple over $\hat{\w{e}}$ in $\mathfrak{J}$ with the witnesses $D$, $h$ and $(\w{b}_{h(n)})_{n\in D}\subseteq \mathfrak{J}\setminus\{0\}$. If $A\subseteq D$ and $\lim_{n\in A}\frac{1}{\|\w{b}_{h(n)}\|_{\mathfrak{J}}}=0$, then there is $\w{x}\in\mathfrak{J}$ such that: 
\[
A\subseteq A_{\w{x}}=\left\{n\in D: |\w{e}^\star_{h(n)}(\w{x})|\geq\frac{1}{\|\w{b}_{h(n)}\|_{\mathfrak{J}}}\right\}.
\]
\end{lemma}

\begin{proof}
The claim is clear if $A\in\Fin$, so assume that $A$ is infinite. Denote:
\[
a=\max\left\{\frac{1}{\|\w{b}_{h(n)}\|_{\mathfrak{J}}}:n\in A\right\}.
\]
Put also: 
\[
A_k=\left\{n\in A:\frac{1}{\|\w{b}_{h(n)}\|_{\mathfrak{J}}}>\frac{1}{k+1}\right\}
\]
for all $k\in\omega$. Observe that $A=\bigcup_{k\in\omega}A_k$ and $\lim_k \max h[A_{k}]=\infty$ (since $A$ is infinite and $h$ is interval-to-one). 

Let $\w{x}$ be such that $x_i=a$ for all $i\leq \max h[A_0]$ and $x_i=\frac{1}{k+1}$ for all $\max h[A_{k}]<i\leq \max h[A_{k+1}]$ and $k\in\omega$. Observe that $\|\w{x}\|_{\mathfrak{J}}=a^2$, as for every $n\in\omega$ and $0\leq p_0<p_1<\ldots<p_n$ we have:
\[
\sum_{i=1}^n \left(x_{p_i}-x_{p_{i-1}}\right)^2\leq\left(\sum_{i=1}^n\left(x_{p_i}-x_{p_{i-1}}\right)\right)^2=\left(x_{p_0}-x_{p_{n}}\right)^2<a^2.
\]
Hence, $\w{x}\in\mathfrak{J}$. We will show that $A\subseteq A_{\w{x}}$. Fix $n\in A$. If $n\in A_0$ then $\frac{1}{\|\w{b}_{h(n)}\|_{\mathfrak{J}}}\leq a$ and $|\w{e}^\star_{h(n)}(\w{x})|=|x_{h(n)}|=a$ (since $h(n)\leq\max h[A_0]$). On the other hand, if $n\in A_{k+1}\setminus A_k$, for some $k\in\omega$, then $\frac{1}{\|\w{b}_{h(n)}\|_{\mathfrak{J}}}\leq \frac{1}{k+1}$ and $|\w{e}^\star_{h(n)}(\w{x})|=|x_{h(n)}|=\frac{1}{k+1}$ (since $h(n)\leq\max h[A_{k+1}]$). Hence, in both cases $n\in A_{\w{x}}$.
\end{proof}

\subsection{Characterization}

Actually, Proposition \ref{0xFin} can be improved:

\begin{proposition}
\label{0xFin-lepsze}
Let $\I$ be an ideal on $\omega$ isomorphic with $\{\emptyset\}\otimes\Fin$. Then there are $\hat{\w{a}}$ and $\hat{\w{a}}'$, both simple over $\hat{\w{e}}$, such that $\CR(c_0,\hat{\w{a}})=\I$ and $\CR(\mathfrak{J},\hat{\w{a}}')=\I$.
\end{proposition}

\begin{proof}
Let $g:\omega^2\to\omega$ be the bijection witnessing that $\I$ is isomorphic with $\{\emptyset\}\otimes\Fin$ and denote $X_k=g[\{k\}\times\omega]$ for all $k\in\omega$. Put $t_k=k$ for all $k\in\omega$ and define the sequence $\w{s}$ by:
\[
s_n=k+1\ \Longleftrightarrow\ n\in X_k
\]
for all $n\in\omega$. Apply Lemma \ref{lem:any_sequence} for $\w{t}$, $\w{s}$ and $X=c_0$ ($\w{t}$, $\w{s}$ and $X=\mathfrak{J}$, respectively) to get $\hat{\w{a}}$ ($\hat{\w{a}}'$, respectively) simple over $\hat{\w{e}}$ with the witnesses $D=\omega$, $h:\omega\to\omega$ given by $h(n)=n+1$ and some sequence $(\w{b}_{n+1})_{n\in\omega}$ such that $\|\w{b}_{n+1}\|_{c_0}=s_n$ for all $n\in\omega$ ($(\w{b}'_{n+1})_{n\in\omega}$ such that $\|\w{b}'_{n+1}\|_{\mathfrak{J}}=s_n$ for all $n\in\omega$, respectively).

In the case of the space $c_0$, using Lemma \ref{general-lem2} we get:
\begin{align*}
\CR(c_0,\hat{\w{a}}) & =\left\{A\subseteq\omega: \sum_{n\in A}\frac{\w{e}_{n+1}}{\left\|\w{b}_{n+1}\right\|_{c_0}}\in c_0\right\}=\cr
& =\left\{A\subseteq\omega: \sum_{n\in A}\frac{\w{e}_{n+1}}{s_n}\in c_0\right\}=\cr
& =\left\{A\subseteq\omega: \lim_{n\in A}\frac{1}{s_n}=0\right\}=\I.   
\end{align*}

In the case of the space $\mathfrak{J}$, we cannot use Lemma \ref{general-lem2}; however, from Lemma \ref{general-lem1} we know that $\CR(\mathfrak{J},\hat{\w{a}}')$ is the ideal generated by $\Fin$ and all sets of the form
    \[
    A_{\w{x}}=\left\{n\in \omega: |\w{e}^\star_{n+1}(\w{x})|\geq\frac{1}{\|\w{b}'_{n+1}\|_{\mathfrak{J}}}\right\},
    \]
for $\w{x}\in \mathfrak{J}$. We need to show that $\CR(\mathfrak{J},\hat{\w{a}}')=\I$.

$\CR(\mathfrak{J},\hat{\w{a}}')\subseteq\I$: It suffices to show that $A_{\w{x}}\in\I$ for all $\w{x}\in \mathfrak{J}$, so fix any such $\w{x}$. Since $\w{x}\in \mathfrak{J}\subseteq c_0$, we have: 
\[
\lim_{n\in A_{\w{x}}}\frac{1}{\|\w{b}'_{n+1}\|_{\mathfrak{J}}}=\lim_{n\in A_{\w{x}}}\frac{1}{s_n}=0.
\]
Hence, $A_{\w{x}}\cap X_k\in\Fin$ for all $k\in\omega$, which means that $A_{\w{x}}\in\I$.

$\CR(\mathfrak{J},\hat{\w{a}}')\supseteq\I$: Fix any $A\in\I$. Then
\[
\lim_{n\in A}\frac{1}{\|\w{b}'_{n+1}\|_{\mathfrak{J}}}=\lim_{n\in A}\frac{1}{s_n}=0.
\]
By Lemma \ref{lem:James}, there is  $\w{x}\in \mathfrak{J}$ such that $A\subseteq A_{\w{x}}$, so $A\in \CR(\mathfrak{J},\hat{\w{a}}')$.
\end{proof}

\begin{theorem}
\label{char:c_0}
The following are equivalent for any nontrivial ideal $\I$ on $\omega$:
\begin{itemize}
    \item[(a)] there is $\hat{\w{a}}$ simple over $\hat{\w{e}}$ such that $\CR(c_0,\hat{\w{a}})=\I$,
    \item[(b)] there is $\hat{\w{a}}$ simple over $\hat{\w{e}}$ such that $\CR(\mathfrak{J},\hat{\w{a}})=\I$,
    \item[(c)] $\I$ is isomorphic to one of the following ideals: $\Fin$, $\Fin\oplus\mathcal{P}(\omega)$ or $\{\emptyset\}\otimes\Fin$.
\end{itemize}
\end{theorem}

\begin{proof}
(c)$\implies$(a) and (c)$\implies$(b): If $\I\cong\Fin$, then actually $\I=\Fin$, so $\CR(c_0,\hat{\w{e}})=\I$ and $\CR(\mathfrak{J},\hat{\w{e}})=\I$. On the other hand, if $\I\cong \Fin\oplus\mathcal{P}(\omega)$ or $\I\cong\{\emptyset\}\otimes\Fin$, then it suffices to apply Theorem \ref{thm-distinguishing} or Proposition \ref{0xFin-lepsze}, respectively. 

(a)$\implies$(c) and (b)$\implies$(c): Let $X=c_0$ or $X=\mathfrak{J}$. Let $\hat{\w{a}}$ be simple over $\hat{\w{e}}$. Then there are $D\subseteq\omega$, an interval-to-one $h:D\to\omega$ and $(\w{b}_{h(n)})\subseteq X\setminus\{0\}$ such that:
\[
S_n(\hat{\w{u}})(\w{x})-S_n(\hat{\w{a}})(\w{x})=
\begin{cases}
    0, & \text{if }n\in\omega\setminus D,\\
    \w{u}_{h(n)}^\star(\w{x})\w{b}_{h(n)}, & \text{if }n\in D,
\end{cases}
\]
for all $n\in \omega$ and $\w{x}\in X$.

Define $B_0=\{n\in D: \|\w{b}_{h(n)}\|_{X}\geq 1\}$ and $B_i=\{n\in D: \frac{1}{k+1}\leq\|\w{b}_{h(n)}\|_{X}<\frac{1}{k}\}$ for all $n\in\omega\setminus\{0\}$. Then $(B_k)$ is a partition of $D$. Denote $T=\bigcup\{B_k: B_k\in\Fin\}$ and $S=\{k\in\omega: B_k\notin\Fin\}$.

Observe that:
\begin{itemize}
    \item[(i)] $\CR(X,\hat{\w{a}})\restriction \omega\setminus D=\Fin\restriction \omega\setminus D$;
    \item[(ii)] $\CR(X,\hat{\w{a}})\restriction B_k=\Fin\restriction B_k$ for every $k\in\omega$;
    \item[(iii)] if $A\subseteq D$ and $A\cap B_k\in\Fin$ for all $k$, then $A\in \CR(X,\hat{\w{a}})$;
    \item[(iv)] $\CR(X,\hat{\w{a}})\restriction T=\mathcal{P}(T)$.
\end{itemize} 
Indeed, in the case of $X=c_0$ all items follow from Lemma \ref{general-lem2}, while in the case of $X=\mathfrak{J}$, item (ii) follows from the fact that $\mathfrak{J}\subseteq c_0$ and items (iii) and (iv) follow from Lemma \ref{lem:James}.

There are four possibilities:
\begin{itemize}
    \item If $\omega\setminus T\in\Fin$, then $\CR(X,\hat{\w{a}})=\mathcal{P}(\omega)$ (by (iv)).
    \item If $S\in\Fin$ and $T\in\Fin$, then $\CR(X,\hat{\w{a}})\cong\Fin$ (by (i) and (ii)).
    \item If $S\in\Fin$, $T\notin\Fin$, but $\omega\setminus T\notin\Fin$, then $\CR(X,\hat{\w{a}})\cong\Fin\oplus\mathcal{P}(\omega)$ (by (i), (ii) and (iv)).
    \item If $S\notin\Fin$, then $\CR(X,\hat{\w{a}})\cong\{\emptyset\}\otimes\Fin$ (by (ii), (iii) and Lemma \ref{lem-isomorphisms}).
\end{itemize} 
Hence, if $\CR(X,\hat{\w{a}})\cong\I$ and $\I$ is nontrivial, then $\I$ is isomorphic to $\Fin$, $\Fin\oplus\mathcal{P}(\omega)$ or $\{\emptyset\}\otimes\Fin$.
\end{proof}

\begin{remark}
\label{remark}
Suppose that $\hat{\w{a}}$ is simple over $\hat{\w{e}}$. Then $\CR(c_0,\hat{\w{a}})$ and $\CR(\mathfrak{J},\hat{\w{a}})$ are not tall, but they can be non-$\bf{F_\sigma}$ (by Theorem \ref{char:c_0}). On the other hand, for each $1\leq p<\infty$, $\CR(\ell_p,\hat{\w{a}})$ has to be $\bf{F_\sigma}$, but it can be tall (by Theorem \ref{char:lp}).
\end{remark}

\section{A Schauder basis for every non-pathological analytic P-ideal}

Results in this section were inspired  by \cite[Proposition 5.3]{PBNF}.

\begin{theorem}
\label{T:allP}
If $\I$ is a non-pathological analytic P-ideal, then there are a Banach space $X\subseteq \mathbb{R}^\omega$ and $\hat{\w{a}}$ simple over $\hat{\w{e}}$ such that $\CR(X,\hat{\w{a}})=\I$.
\end{theorem}

\begin{proof}
Since $\I$ is a non-pathological analytic P-ideal, there is a non-pathological lsc submeasure $\phi$ such that $\I=\Exh(\phi)$. Without loss of generality we may assume that $\supp(\phi)=\omega$ (it suffices to consider $\phi'$ given by 
\[
\phi'(A)=\phi(A)+\sum_{n\in A\setminus\supp(\phi)}\frac{1}{2^n},
\]
for all $A\subseteq\omega$ and observe that $\Exh(\phi)=\Exh(\phi')$). 

Let $\mathcal{M}_\phi$ be the family of all measures on $\omega$ such that $\supp(\mu)\in\Fin$ and $\mu(B)\leq\phi(B)$ for all $B\subseteq\omega$. Define:
\[
\|\w{x}\|_\phi=|x_0|+\sup_{\mu\in\mathcal{M}_\phi}\left(\sum_{k\in\supp(\mu)}(k+1)\mu(\{k\})|x_{k+1}|\right)
\]
for all $\w{x}\in\mathbb{R}^\omega$ and let:
\[
X=\left\{x\in\mathbb{R}^\omega: \lim_{n\to\infty}\left\|\w{x}-\sum_{i=0}^n \w{e}_i^\star(\w{x})\w{e}_i\right\|_\phi=0\right\}=\left\{x\in\mathbb{R}^\omega: \lim_{n\to\infty}\left\|\sum_{i>n} \w{e}_i^\star(\w{x})\w{e}_i\right\|_\phi=0\right\}.
\]
Then $\|\cdot\|_\phi$ is a norm on $\mathbb{R}^\omega$ (as $\supp(\phi)=\omega$) making $(X,\|\cdot\|_\phi)$ a Banach space. Moreover, $\hat{\w{e}}$ is an unconditional Schauder basis of $X$ (by \cite[Proposition 3.1.3]{AK}, since if $|x_n|\leq|y_n|$ for all $n$, then $\|\w{x}\|_\phi\leq\|\w{y}\|_\phi$).

For each $n\in\omega$ inductively find $p_n\in\mathbb{R}$ such that $\|\sum_{i=0}^n p_i\w{e}_i\|_\phi=n+1$ and let $\hat{\w{a}}=\left(\w{a}_n,\w{a}_n^\star\right)_n$ be given by $\w{a}_n = \sum_{i=0}^n p_i\w{e}_i $ and $\w{a}_n^\star=\frac{\w{e}_n^\star}{p_n} - \frac{\w{e}_{n+1}^\star}{p_{n+1}}$ for all $n\in\omega$. Clearly, $\w{a}^\star_n(\w{a}_m)=\delta_{nm}$ for all $n,m\in\omega$.

Observe that:
\[
S_n(\hat{\w{e}})(\w{x})-S_n(\hat{\w{a}})(\w{x})=\w{e}_{n+1}^\star(\w{x})\w{a}_{n},
\]
for every $\w{x}\in X$. In particular, $\hat{\w{a}}$ is simple over $\hat{\w{e}}$ with $h=\on{id}+1$ witnessing it. Hence, applying Lemma \ref{general-lem2}, we have:
\begin{align*}
\CR(X,\hat{\w{a}}) & =\left\{A\subseteq\omega: \sum_{k\in A}\frac{\w{e}_{k+1}}{\left\|\w{a}_k\right\|_\phi}\in X\right\}=\left\{A\subseteq\omega: \lim_{n\to\infty}\left\|\sum_{k\in A\setminus n} \frac{\w{e}_{k+1}}{k+1}\right\|_\phi=0\right\}=\cr
 & =\left\{A\subseteq\omega: \lim_{n\to\infty}\sup_{\mu\in\mathcal{M}_\phi}\left(\sum_{k\in\supp(\mu)\cap(A\setminus n)}(k+1)\mu(\{k\})\frac{1}{k+1}\right) =0\right\}=\cr
 & =\left\{A\subseteq\omega: \lim_{n\to\infty}\sup_{\mu\in\mathcal{M}_\phi}\mu(A\setminus n) =0\right\}= \left\{A\subseteq\omega: \lim_n\phi(A\setminus n) =0\right\}=\Exh(\phi)=\I.  
\end{align*}
\end{proof}

\begin{corollary}
There are a Banach space $X\subseteq \mathbb{R}^\omega$ and $\hat{\w{a}}$ simple over $\hat{\w{e}}$ such that $\CR(X,\hat{\w{a}})=\I_d$. In particular,  $\CR(X,\hat{\w{a}})\not\cong\CR(c_0,\hat{\w{w}})$ and $\CR(X,\hat{\w{a}})\not\cong\CR(\ell_p,\hat{\w{w}})$, for all $\hat{\w{w}}$ simple over $\hat{\w{e}}$ and all $1\leq p<\infty$.
\end{corollary}

\begin{proof}
The existence of $X\subseteq \mathbb{R}^\omega$ and $\hat{\w{a}}$ simple over $\hat{\w{e}}$ such that $\CR(X,\hat{\w{a}})=\I_d$ follows from Theorem \ref{T:allP}, since $\I_d$ is non-pathological ($\I_d=\Exh(\sup_{n\in\omega}\mu_n)$ for measures $\mu_n:\mathcal{P}(\omega)\to[0,\infty)$ given by $\mu_n(A)=\frac{|A\cap(n+1)|}{n+1}$, for all $n\in\omega$ and $A\subseteq\omega$).

The fact that $\I_d\not\cong\CR(\ell_p,\hat{\w{w}})$, for all $\hat{\w{w}}$ simple over $\hat{\w{e}}$ and all $1\leq p<\infty$, follows now from Remark \ref{remark}, since $\I_d$ is not $\bf{F_\sigma}$.

To see that $\I_d\not\cong\CR(c_0,\hat{\w{w}})$, for all $\hat{\w{w}}$ simple over $\hat{\w{e}}$, note that $\CR(c_0,\hat{\w{w}})$ is not tall (by Remark \ref{remark}), while $\I_d$ is tall. 
\end{proof}

\section{Avoiding $c_{00}$}\label{S:avoidingC}

Examples considered so far were based on vectors being elements of $c_{00}$. In this Section, we propose an alternative approach, which involves introducing more zeros at the beginning of the sequence under consideration. As the case of James’ space appears more technical, we restrict our discussion in this Section to the spaces $c_0$ and $\ell_p$.

We also use this section to present alternative method of proving $\w{a}^\star_n(\w{a}_m)=\delta_{nm}.$ Observe that the condition follows from $\w{x}=\sum_{\I,n} \w{a}^\star_n(\w{x})\w{a}_n$ and the existence of some sequence of continuous functionals $\w{b}_n^\star$ such that $\w{a}_n\in\on{Ker}\w{b}^\star_m$ iff $n\neq m$ (since the latter condition is equivalent to $\w{a}_n\notin \overline{\on{span}\{\w{a}_m:n\neq n\}}$ for all $n$). Within this section it will be enough to consider $\w{b}^\star_n=\w{e}_n^\star$ in all examples except \ref{ex:simpoversimp}. 

\begin{definition}\label{d:Avoidc0}
For $\w{z}\in(0,\infty)^\omega$ we define $\w{a}_n(\w{z})=\w{z}\cdot \chi_{[n,\infty)}=\sum_{k=n}^\infty z_n\w{e}_k$ for all $n\in\omega$ and set $\w{a}_0^\star(\w{z})=\frac{1}{z_0}\w{e}_0^\star$ and $\w{a}_n^\star(\w{z})=\frac{1}{z_n}\w{e}_n^\star - \frac{1}{z_{n-1}}\w{e}_{n-1}^\star$ for all $n\geq 1$.
\end{definition}

Note that in this case, pointwise convergence determines the form of the functionals $\w{a}_n^\star(\w{z})$.

\begin{lemma}
\label{lem:z}
Let $\w{z}\in(0,\infty)^\omega$. Then $\hat{\w{a}}(\w{z})$ is simple over $\hat{\w{e}}$ with witnesses $D=\omega$, $h=\on{id}$ and $\w{b}_n=-\sum_{m>n}\frac{z_m}{z_n}\w{e}_m$ for all $n\in\omega$. 
\end{lemma}

\begin{proof}
Observe that if $\w{x}\in\mathbb{R}^\omega$, then for $m\leq n$ we get 
$\w{e}_m^\star(S_n(\hat{\w{a}}(\w{z}))(\w{x}))=\w{e}_m^\star(\w{x})$, while for $m>n$ we have: 
\begin{align*}
\w{e}_m^\star\left(S_n(\hat{\w{a}}(\w{z}))(\w{x})\right) & =\frac{\w{e}_0^\star(\w{x})z_m}{z_0}+\sum_{i=1}^n \left(\frac{1}{z_i}\w{e}_i^\star(\w{x}) - \frac{1}{z_{i-1}}\w{e}_{i-1}^\star(\w{x})\right)z_m=\cr
& =\left(\frac{\w{e}_0^\star(\w{x})}{z_0}+\sum_{i=1}^n \left(\frac{1}{z_{i}}\w{e}_i^\star(\w{x}) - \frac{1}{z_{i-1}}\w{e}_{i-1}^\star(\w{x})\right) \right)z_m=\frac{z_m\w{e}^\star_n(\w{x})}{z_n}.
\end{align*}
Hence:
\[
S_n(\hat{\w{e}})(\w{x})-S_n(\hat{\w{a}}(\w{z}))(\w{x})=-\w{e}^\star_n(\w{x})\sum_{m>n} \frac{z_m}{z_n} \w{e}_m. 
\]
In particular, $\hat{\w{a}}(\w{z})$ is simple over $\hat{\w{e}}$ as witnessed by $D=\omega$, $h=\on{id}$ and $\w{b}_n=-\sum_{m>n}\frac{z_m}{z_n}\w{e}_m$. 
\end{proof}

The first specific example of this approach will be $\w{z}=(\frac{1}{n+1})_{n \in \omega}$. Note that $\w{z}\notin\ell_1$, so also $\w{a}_n(\w{z})\notin \ell_1$ for all $n\in\omega$ and we cannot consider $\hat{\w{a}}(\w{z})$ in the space $\ell_1$.

\begin{proposition}
Let $\w{z}=(\frac{1}{n+1})_{n \in \omega}$. Then $\CR(c_0,\hat{\w{a}}(\w{z}))=\CR(\mathfrak{J},\hat{\w{a}}(\w{z}))=\Fin$ and $\CR(\ell_p,\hat{\w{a}}(\w{z}))=\I_{1/n}$ for all $1<p<\infty$.
\end{proposition}

\begin{proof}
Let $(\w{b}_n)$ be as in Lemma \ref{lem:z}.

Note that $\|\w{b}_n\|_{c_0}=\|-\sum_{m>n}\frac{z_m}{z_n}\w{e}_m\|_{c_0}=\frac{n+1}{n+2}$ for all $n\in\omega$. Thus, using Lemmas \ref{general-lem2} and \ref{lem:z}, we have:
\[
\CR(c_0,\hat{\w{a}}(\w{z}))=\left\{A\subseteq\omega:\sum_{n\in A}\frac{\w{e}_n}{\|\w{b}_n\|_{c_0}}\in c_0\right\}=\left\{A\subseteq\omega:\sum_{n\in A}\frac{n+2}{n+1}\w{e}_n\in c_0\right\}=\Fin.
\]

Similarly, since $\|\w{b}_n\|_{\mf{J}}=\|-\sum_{m>n}\frac{z_m}{z_n}\w{e}_m\|_{\mf{J}}=\sqrt{2}\frac{n+1}{n+2}$ for all $n\in\omega$, using Lemmas \ref{general-lem1} and \ref{lem:z} we get $\CR(\mathfrak{J},\hat{\w{a}}(\w{z}))=\Fin$ (as $A_{\w{x}}$ is finite for every $\w{x}\in\mf{J}$).

If now $1<p<\infty$, then:
\[
\|\w{b}_n\|_{\ell_p}=\left\|-\sum_{m>n}\frac{z_m}{z_n}\w{e}_m\right\|_{\ell_p}=(n+1)\left(\sum_{m>n}\frac{1}{(m+1)^p}\right)^{1/p}.
\]
Estimating $\sum_{m>n}\frac{1}{(m+1)^p}$ by $\int_{n+1}^\infty \frac{1}{x^p} dx $ yields $\frac{(n+1)^{1-p}}{p-1}$, hence we get:
\[
\|\w{b}_n\|_{\ell_p}\approx(n+1)\left(\frac{(n+1)^{1-p}}{p-1}\right)^{1/p}=\left(\frac{n+1}{p-1}\right)^{1/p}.
\]
Using Lemmas \ref{general-lem2} and \ref{lem:z} once again, we have:
\[
\CR(\ell_p,\hat{\w{a}}(\w{z}))=\left\{A\subseteq\omega: \sum_{n\in A}\frac{\w{e}_n}{\|\w{b}_n\|_{\ell_p}}\in\ell_p\right\}=\left\{A\subseteq\omega: \sum_{n\in A}\frac{p-1}{n+1}<\infty\right\}=\I_{1/n}.
\]
\end{proof}

Let us now switch to the case of arbitrary $\w{z}\in (0,\infty)^\omega\cap X$, and firstly consider it within the space $X=c_0$. 

\begin{proposition}
\label{z:c_0}
Let $\w{z}\in (0,\infty)^\omega\cap c_0$. Then 
\[
\CR(c_0, \hat{\w{a}}(\w{z}))=\left\{A\subseteq\omega: \lim_{n\in A}\frac{1}{\alpha_n(\w{z})}=0\right\},
\]
where $\alpha_n(\w{z})=\sup_{m>n}\frac{z_m}{z_n}$ for all $n\in\omega$. Moreover:
\begin{itemize}
    \item[(a)] if  $(\alpha_n(\w{z}))$ is bounded (in particular, if $\w{z}$ is monotone), then $\CR(c_0,\hat{\w{a}}(\w{z}))=\Fin$;
    \item[(b)] if there is an infinite $E\subseteq\omega$ such that $(\alpha_n(\w{z}))_{n\in\omega\setminus E}$ is bounded and $\lim_{n\in E}\alpha_n(\w{z})=\infty$, then $\CR(c_0,\hat{\w{a}}(\w{z}))$ is the ideal generated by $\Fin\cup\{E\}$; 
    \item[(c)] if there is a partition $(D_k)$ of $\omega$ into infinite sets such that $\lim_{k\to\infty}\inf_{n \in D_k} \alpha_n(\w{z})=\infty$ and $\sup_{n \in D_k} \alpha_n(\w{z}) < \infty$ for every $k\in\omega$, then:
    \[
    \CR(c_0,\hat{\w{a}}(\w{z}))=\{A\subseteq\omega:A\cap D_k\in\Fin\text{ for all }k\in\omega\}.\]
\end{itemize}
\end{proposition}

\begin{proof}
Let $(\w{b}_n)$ be as in Lemma \ref{lem:z}. Since $\|\w{b}_n\|_{c_0}=\alpha_n(\w{z})$ for all $n\in\omega$, using Lemmas \ref{general-lem2} and \ref{lem:z}, we have:
\[
\CR(c_0,\hat{\w{a}}(\w{z}))=\left\{A\subseteq\omega:\sum_{n\in A}\frac{\w{e}_n}{\|\w{b}_n\|_{c_0}}\in c_0\right\}=\left\{A\subseteq\omega: \lim_{n\in A}\frac{1}{\alpha_n(\w{z})}=0\right\}.
\]
Items (a), (b) and (c) are clear.
\end{proof}

\begin{lemma}
\label{L:alphy}
A sequence $(y_n) \in(0,\infty)^\omega$ is of the form $(\alpha_n(\w{z}))$ for some $\w{z}\in (0,\infty)^\omega\cap c_0$ if and only if $\prod \{y_n : y_n <1\}=0$.
\end{lemma}

\begin{proof}
Assume that there is $\w{z}\in c_0$ such that $y_n=\alpha_n(\w{z})$ for all $n\in\omega$. Inductively find $(k_n)$ such that $k_0=0$ and $y_{k_n}=\frac{z_{k_{n+1}}}{z_{k_n}}$ for all $n\in\omega$ (in particular, $y_0=\frac{z_{k_1}}{z_0}$). Observe that $z_{k_{n+1}}=z_0\cdot \prod \{y_{k_i} : i\leq n\}$. Since $\w{z}\in c_0$, we get that $\prod \{y_{k_i} : i\in\omega\}=0$. In particular, $\prod \{y_{k_i} : y_{k_i}<1\}=0$, so also $\prod \{y_n : y_n <1\}=0$.

To see the other implication, let $(n_i)$ be the increasing enumeration of the set $\{n : y_n <1\}$. Define $z_{n_0}=1$, $z_{n_{i+1}}=y_{n_i} z_{n_i}$, $z_n = z_{n_0}/y_n$ for $n<n_0$ and $z_n=z_{n_{i+1}}/y_n$ for $n\in (n_i, n_{i+1})$. Clearly, $\w{z}\in (0,\infty)^\omega\cap c_0$ and $y_n=\alpha_n(\w{z})$ for all $n\in\omega$.
\end{proof}

\begin{corollary}
\label{char:c_0-2}
For every ideal $\I$ on $\omega$ we have: $\CR(c_0,\hat{\w{a}}(\w{z}))=\I$ for some $\w{z}\in (0,\infty)^\omega\cap c_0$ if and only if $\I$ is isomorphic to $\Fin, \Fin\oplus\mathcal{P}(\omega)$ or $\{\emptyset\}\otimes\Fin$.
\end{corollary}

\begin{proof}
Notice that for each $\w{z}\in (0,\infty)^\omega\cap c_0$ the sequence $(\alpha_n(\w{z}))$ is either bounded or there is an infinite $E\subseteq\omega$ such that $(\alpha_n(\w{z}))_{n\in\omega\setminus E}$ is bounded and $\lim_{n\in E}\alpha_n(\w{z})=\infty$ or there is a partition $(D_k)$ of $\omega$ into infinite sets such that $\lim_{k\to\infty}\inf_{n \in D_k} \alpha_n(\w{z})=\infty$ and $\sup_{n \in D_k} \alpha_n(\w{z}) < \infty$ for every $k\in\omega$. Thus, thanks to Proposition \ref{z:c_0}, the ideal $\CR(c_0,\hat{\w{a}}(\w{z}))$ is equal to $\Fin$, generated by $\Fin\cup\{E\}$, for some infinite $E\subseteq\omega$ (in this case it is isomorphic with $\Fin\oplus\mathcal{P}(\omega)$), or of the form $\{A\subseteq\omega:A\cap D_k\in\Fin\text{ for all }k\in\omega\}$, for some partition $(D_k)$ of $\omega$ into infinite sets (in this case it is isomorphic with $\{\emptyset\}\otimes\Fin$). This shows one implication.

If now $\I$ is isomorphic to $\Fin, \Fin\oplus\mathcal{P}(\omega)$ or $\{\emptyset\}\otimes\Fin$, then we can easily find a sequence $(y_n)$ such that $\prod \{y_n : y_n <1\}=0$ and $\I=\{A\subseteq\omega:\lim_{n\in A}\frac{1}{y_n}=0\}$. By Lemma \ref{L:alphy}, there is $\w{z}\in (0,\infty)^\omega\cap c_0$ such that $y_n=\alpha_n(\w{z})$ for all $n$. In particular, $\I=\CR(c_0,\hat{\w{a}}(\w{z}))$ (by Proposition \ref{z:c_0}).
\end{proof}
   
Let us now switch to the case of the space $\ell_p$. 

\begin{proposition}
\label{z:l_p}
Let $\w{z}\in (0,\infty)^\omega\cap \ell_p$. Then $\CR(\ell_p,\hat{\w{a}}(\w{z}))=\I_{\frac{1}{\beta_n(\w{z})}}$, where $\beta_n(\w{z})=\sum_{m>n}\left|\frac{z_m}{z_n}\right|^p$ for all $n\in\omega$.
\end{proposition}

\begin{proof}
Let $(\w{b}_n)$ be as in Lemma \ref{lem:z}. Since $\|\w{b}_n\|_{\ell_p}=(\beta_n(\w{z}))^{1/p}$ for all $n\in\omega$, using Lemmas \ref{general-lem2} and \ref{lem:z}, we have:
\[
\CR(\ell_p,\hat{\w{a}}(\w{z}))=\left\{A\subseteq\omega:\sum_{n\in A}\frac{\w{e}_n}{\|\w{b}_n\|_{\ell_p}}\in \ell_p\right\}=\left\{A\subseteq\omega: \sum_{n\in A}\frac{1}{\beta_n(\w{z})}<\infty\right\}=\I_{\frac{1}{\beta_n(\w{z})}}.
\]
\end{proof}

\begin{remark}
By Theorem $\ref{char:c_0}$ and Corollary \ref{char:c_0-2} we know that in the space $c_0$, for every $\hat{\w{a}}$ simple over $\hat{\w{e}}$ there is some $\w{z}\in (0,\infty)^\omega\cap c_0$ such that $\CR(c_0, \hat{\w{a}})=\CR(c_0, \hat{\w{a}}(\w{z}))$. In the spaces $\ell_p$ the situation is different: by Corollary \ref{cor:non-summable} and Proposition \ref{z:l_p}, there is $\hat{\w{a}}$ simple over $\hat{\w{e}}$ such that $\CR(\ell_p, \hat{\w{a}})\neq\CR(\ell_p, \hat{\w{a}}(\w{z}))$ for all $\w{z}\in (0,\infty)^\omega\cap \ell_p$.
\end{remark}

Next example shows that it is possible to have three $\Fin$-Schauder bases $\hat{\w{a}}$, $\hat{\w{b}}$ and $\hat{\w{c}}$ of $c_0$ such that $\hat{\w{a}}$ is simple over $\hat{\w{b}}$, $\hat{\w{b}}$ is simple over $\hat{\w{c}}$, but $\hat{\w{a}}$ is not simple over $\hat{\w{c}}$.

\begin{example}
\label{ex:simpoversimp}
Let $z_n=\frac{1}{2^n}$ for all $n\in\omega$ and consider $\hat{\w{a}}(\w{z})$. Then $\alpha_n(\w{z})=2$ for all $n\in\omega$, so $\CR(c_0, \hat{\w{a}}(\w{z}))=\Fin$ (by Proposition \ref{z:c_0}) and $\hat{\w{a}}(\w{z})$ is a $\Fin$-Schauder basis of $c_0$. Moreover, $\hat{\w{a}}(\w{z})$ is simple over $\hat{\w{e}}$ with the witnesses $D=\omega$, $h=\on{id}$ and $\w{b}_n=-\sum_{m>n}\frac{z_m}{z_n}\w{e}_m$ for all $n\in\omega$ (by Lemma \ref{lem:z}).

Now we introduce the third $\Fin$-Schauder basis of $c_0$. Let $\w{v}_0=\w{a}_0(\w{z})$ and $\w{v}_n=\w{a}_0(\w{z})+\sum_{i=1}^n 2^{i-1}\w{a}_i(\w{z})$ for $n\geq 1$. Put also $\w{v}_0^\star= \w{a}^\star_0(\w{z})-\w{a}^\star_1(\w{z})$ and $\w{v}_n^\star= \frac{1}{2^{n-1}}\w{a}^\star_n(\w{z})-\frac{1}{2^n}\w{a}^\star_{n+1}(\w{z})$ for $n\geq 1$. Then $\w{v}^\star_n(\w{v}_m)=\delta_{nm}$ for all $n,m\in\omega$ (since $\w{a}^\star_n(\w{a}_m)=\delta_{nm}$ for all $n,m\in\omega$). Note that $\w{v}_n=\w{1}_{n}^\frown \w{z}$, where $\w{1}_n$ is the finite sequence consisting of $n$ ones. Moreover, by Definition \ref{d:Avoidc0}, $\w{v}^\star_0=2\w{e}^\star_0-2\w{e}^\star_1$ and $\w{v}_n^\star= -\w{e}^\star_{n-1}+3\w{e}_n^\star-2\w{e}_{n+1}^\star$ for $n \geq 1$. 

Observe that:
\begin{align*}
S_n(\hat{\w{v}})(\w{x}) & =\w{a}_0(\w{z})\left(\w{a}^\star_0(\w{z})(\w{x})-\frac{\w{a}^\star_{n+1}(\w{z})(\w{x})}{2^n}\right)+\sum_{i=1}^n \w{a}_i(\w{z})\left(\w{a}^\star_i(\w{z})(\w{x})-\frac{\w{a}^\star_{n+1}(\w{z})(\w{x})}{2^{n-i+1}}\right)=\cr
& = S_n(\hat{\w{a}}(\w{z}))(\w{x})-\w{a}^\star_{n+1}(\w{z})(\w{x})\left(\frac{\w{a}_0(\w{z})}{2^n}+\sum_{i=1}^n \frac{\w{a}_i(\w{z})}{2^{n-i+1}}\right).
\end{align*}
Hence, $\hat{\w{v}}$ is simple over $\hat{\w{a}}(\w{z})$ as witnessed by $D'=\omega$, $h'=\on{id}+1$ and $\w{d}_{n+1}=\frac{\w{a}_0(\w{z})}{2^n}+\sum_{i=1}^n \frac{\w{a}_i(\w{z})}{2^{n-i+1}}$ for all $n\in\omega$. On the other hand, $\hat{\w{v}}$ is not simple over $\hat{\w{e}}$, since:
\begin{align*}
S_n(\hat{\w{e}})(\w{x}) - S_n(\hat{\w{v}})(\w{x}) & = S_n(\hat{\w{e}})(\w{x}) - \left(S_n(\hat{\w{a}}(\w{z}))(\w{x}) - \w{a}^\star_{n+1}(\w{z})(\w{x})\w{d}_{n+1}\right)=\cr
& =S_n(\hat{\w{e}})(\w{x}) - \left(S_n(\hat{\w{e}})(\w{x})+\w{e}^\star_{n}(\w{x})\w{b}_{n}\right)+ \w{a}^\star_{n+1}(\w{z})(\w{x})\w{d}_{n+1}=\cr
& =-\w{e}^\star_{n}(\w{x})\w{b}_{n}+\left(2^{n+1}\w{e}^\star_{n+1}(\w{x})-2^n\w{e}^\star_{n}(\w{x})\right)\w{d}_{n+1}=\cr
& =\w{e}^\star_{n+1}(\w{x})2^{n+1}\w{d}_{n+1}-\w{e}^\star_{n}(\w{x})\left(\w{b}_{n}+2^n\w{d}_{n+1}\right).
\end{align*}

Before showing that $\hat{\w{v}}$ is a $\Fin$-Schauder basis of $c_0$, let us make some calculations. Observe that: 
\[
\w{e}^\star_j(\w{d}_{n+1})=\w{e}^\star_j\left(\frac{\w{a}_0(\w{z})}{2^n}+\sum_{i=1}^n \frac{\w{a}_i(\w{z})}{2^{n-i+1}}\right)=\frac{1}{2^j}\left(\frac{1}{2^n}+\sum_{i=1}^n\frac{1}{2^{n-i+1}}\right)=\frac{1}{2^j}\cdot 1=\frac{1}{2^j}
\]
for all $j>n$ and that:
\[
\w{e}^\star_j(\w{d}_{n+1})=\w{e}^\star_j\left(\frac{\w{a}_0(\w{z})}{2^n}+\sum_{i=1}^n \frac{\w{a}_i(\w{z})}{2^{n-i+1}}\right)=\frac{1}{2^j}\left(\frac{1}{2^n}+\sum_{i=1}^j\frac{1}{2^{n-i+1}}\right)=\frac{1}{2^{j}}\cdot\frac{1}{2^{n-j}}=\frac{1}{2^{n}}
\]
for all $j\leq n$. 

Now we will show that $\CR(\hat{\w{v}},c_0)=\Fin$ (i.e., that $\hat{\w{v}}$ is a $\Fin$-Schauder basis of $c_0$). Since $\hat{\w{v}}$ is simple over $\hat{\w{a}}(\w{z})$, from Lemma \ref{general-lem1} we know that $\CR(\hat{\w{v}},c_0)$ is generated by $\Fin$ and sets of the form $A_{\w{x}}=\{n\in\omega: |\w{a}^\star_{n+1}(\w{z})(\w{x})|\geq \frac{1}{\|\w{d}_{n+1}\|_{c_0}}\}$ for $\w{x}\in c_0$. By the above calculations, $\|\w{d}_{n+1}\|_{c_0}=\frac{1}{2^{n}}$ for all $n\in\omega$, which gives us:
\begin{align*}
A_{\w{x}} & =\left\{n\in\omega: |\w{a}^\star_{n+1}(\w{z})(\w{x})|\geq \frac{1}{\|\w{d}_{n+1}\|_{c_0}}\right\}=\left\{n\in\omega: |2^{n+1}\w{e}^\star_{n+1}(\w{x})-2^{n}\w{e}^\star_{n}(\w{x})|\geq 2^n\right\}\subseteq\cr
& \subseteq \left\{n\in\omega: |2\w{e}^\star_{n+1}(\w{x})|\geq \frac{1}{2}\right\}\cup \left\{n\in\omega: |\w{e}^\star_{n}(\w{x})|\geq \frac{1}{2}\right\}\in\Fin
\end{align*}
for all $\w{x}\in c_0$. Thus, $\CR(\hat{\w{v}},c_0)=\Fin$.
\end{example}

\begin{remark}
For completeness, fix any $p\in[1,\infty)$ and observe that $\hat{\w{a}}(\w{z})$, for $\w{z}$ considered in Example \ref{ex:simpoversimp} (i.e., $z_n=\frac{1}{2^n}$ for all $n\in\omega$), is a $\Fin$-Schauder basis of $\ell_p$ (by Proposition \ref{z:l_p}, since $\beta_n(\w{z})=\frac{1}{2^p-1}$ for all $n\in\omega$), while $\CR(\hat{\w{v}},\ell_p)=\I_{1/n}$, where $\hat{\w{v}}$ is as in Example \ref{ex:simpoversimp}. Indeed, from Lemma \ref{general-lem1} we know that $\CR(\hat{\w{v}},\ell_p)$ is generated by $\Fin$ and sets of the form $A_{\w{x}}=\{n\in\omega: |\w{a}^\star_{n+1}(\w{z})(\w{x})|\geq \frac{1}{\|\w{d}_{n+1}\|_{\ell_p}}\}$ for $\w{x}\in \ell_p$. 

By the calculations made in Example \ref{ex:simpoversimp}, we have:
\begin{align*}
\|\w{d}_{n+1}\|_{\ell_p} & =\left(\sum_{i=0}^n \frac{1}{2^{np}}+\sum_{i=n+1}^\infty\frac{1}{2^{ip}}\right)^{1/p}=\cr
& =\left(\frac{n+1}{2^{np}}+\frac{1}{2^{(n+1)p}}\cdot\frac{1}{1-\frac{1}{2^{p}}}\right)^{1/p}=\cr
&=\frac{1}{2^n}\left(n+1+\frac{1}{2^p-1}\right)^{1/p}.
\end{align*}
In particular, $\frac{(n+1)^{1/p}}{2^n}\leq\|\w{d}_{n+1}\|_{\ell_p}\leq\frac{(n+2)^{1/p}}{2^n}$. 

Thus, if $\w{x}\in\ell_p$, then:
\begin{align*}
A_{\w{x}} & =\left\{n\in\omega: |\w{a}^\star_{n+1}(\w{z})(\w{x})|\geq \frac{1}{\|\w{d}_{n+1}\|_{\ell_p}}\right\}\subseteq\left\{n\in\omega: |2^{n+1}\w{e}^\star_{n+1}(\w{x})-2^{n}\w{e}^\star_{n}(\w{x})|\geq \frac{2^n}{(n+2)^{1/p}}\right\}\subseteq\cr
& \subseteq \left\{n\in\omega: |2\w{e}^\star_{n+1}(\w{x})|\geq \frac{1}{2(n+2)^{1/p}}\right\}\cup \left\{n\in\omega: |\w{e}^\star_{n}(\w{x})|\geq \frac{1}{2(n+2)^{1/p}}\right\}\in\I_{1/n}.
\end{align*}
This shows that $\CR(\hat{\w{v}},\ell_p)\subseteq\I_{1/n}$. 

On the other hand, if $A\in \I_{1/n}$, then put $A_i=A\cap\{2n+i:n\in\omega\}$ and define:
\[
    x^i_n =
    \begin{cases*}
    \frac{1}{(n+1)^{1/p}},     & if $n\in A_i$, \\
    0, & otherwise,  \\
    \end{cases*}
\]
for $i=0,1$ and $n\in\omega$. Note that $\w{x}^0$ and $\w{x}^1$ belong to $\ell_p$ (as $A\in\I_{1/n}$) and that if $|\w{e}^\star_{n}(\w{x}^i)|\geq \frac{1}{(n+1)^{1/p}}$ for some $n\in\omega$ and $i=0,1$, then $\w{e}^\star_{n+1}(\w{x}^i)=0$. We have:
\begin{align*}
A & =\bigcup_{i=0}^1 A_i=\bigcup_{i=0}^1 \left\{n\in\omega: |\w{e}^\star_{n}(\w{x}^i)|\geq \frac{1}{(n+1)^{1/p}}\right\}\subseteq\cr
&\subseteq\bigcup_{i=0}^1\left\{n\in\omega: |2^{n+1}\w{e}^\star_{n+1}(\w{x}^i)-2^{n}\w{e}^\star_{n}(\w{x}^i)|\geq \frac{2^n}{(n+1)^{1/p}}\right\}\subseteq\cr
& \subseteq\bigcup_{i=0}^1\left\{n\in\omega: |\w{a}^\star_{n+1}(\w{z})(\w{x}^i)|\geq \frac{1}{\|\w{d}_{n+1}\|_{\ell_p}}\right\}=A_{\w{x}^0}\cup A_{\w{x}^1}\in\CR(\hat{\w{v}},\ell_p).
\end{align*}
Hence, $\CR(\hat{\w{v}},\ell_p)\supseteq\I_{1/n}$.
\end{remark}

\section{Ideal version of FDD}

In this subsection we follow some ideas from \cite{ARTT}, where the Authors noted that every Banach space with FDD has the $\I$-Schauder bases for some ideal $\I$, and used this observation to give an example of a space with $\I$-Schauder basis, but without a standard one. Recall that a Banach space $X$ has finitely dimensional decomposition (FDD) if there exists a sequence of its finite-dimensional subspaces $(X_n)$ such that for every $\w{x} \in X$ there exists a unique sequence $(\w{x}_n)$ with $\w{x}_n \in X_n$ for all $n \in \omega$ and such that $\w{x}=\sum_{n \in \omega} \w{x}_n$.

The definition of $\I$-FDD is straightforward:
\begin{definition}
Let $\I$ be an ideal on $\omega$. A Banach space $X$ is said to have $\I$-FDD if there exists a sequence of its finite-dimensional subspaces $(X_n)$ such that for every $\w{x} \in X$ there exists a unique sequence $(\w{x}_n)$ with $\w{x}_n \in X_n$ for all $n \in \omega$ and such that:
\[
\forall_{\varepsilon>0}\ \left\{n\in\omega:\left\|\w{x}-\sum_{i=0}^n \w{x}_i\right\|_X>\varepsilon\right\}\in\I.
\]
\end{definition}

It is known (see \cite{Szar}) that there exists a space with FDD but without a Schauder basis. In the case of ideal versions the situation is a bit different.

\begin{proposition}
\label{FDDprop1}
Let $X$ be a Banach space. If $X$ has $\I$-FDD for some ideal $\I$ on $\omega$, then $X$ has a $\J$-Schauder basis for some ideal $\J$ on $\omega$. Moreover, if $\I$ is non-tall, then also $\J$ is non-tall.
\end{proposition}

\begin{proof}
Let $(X_n)$ be the sequence of finite-dimensional subspaces of $X$ from the difinition of $\I$-FDD. Fix sequences $(\w{u}_n)$ and $0=k_0<k_1<k_2\ldots$ such that $\{\w{u}_i : i \in [k_{n}, k_{n+1})\}$ is a basis of $X_n$. Define $\J$ by:
\[
A \in \J \Longleftrightarrow \{ n \in \omega : k_{n+1}-1 \in A \} \in \I.
\]
Note that if $\I$ is non-tall, then there is some infinite $E\subseteq\omega$ such that $\I\restriction E=\Fin\restriction E$. In this case also $\J$ is non-tall, as $\J\restriction\{n\in\omega:k_{n+1}-1 \in E\}=\Fin\restriction\{n\in\omega:k_{n+1}-1 \in E\}$.
    
Take $x \in X$ and find unique $x_n \in X_n$, for $n\in\omega$, such as in the definition of $\I$-FDD. For every $n\in\omega$ there are also unique $\alpha_i$, for $k_n\leq i<k_{n+1}$, such that $x_n = \sum_{i=k_n}^{k_{n+1}-1}\alpha_i \w{u}_i$. Fix any $\varepsilon>0$ and denote $A_\varepsilon = \{ j \in \omega : \|\w{x}-\sum_{i=0}^{j}\alpha_i \w{u}_i \|> \varepsilon\}$. We need to show that $A_\varepsilon \in\J$. Observe that:
    \[
    \{ n \in \omega : k_{n+1}-1 \in A_\varepsilon \} = \left\{ n \in \omega : \left\|\w{x}-\sum_{i=0}^{k_{n+1}-1}\alpha_i \w{u}_i \right\|> \varepsilon \right\} = \left\{ n \in \omega : \left\|\w{x}-\sum_{i=0}^{n}\w{x}_i \right\|> \varepsilon \right\} \in \I,
    \]
thus $A_\varepsilon \in \J$.
\end{proof}

\begin{remark}
Note that if $\I$ is an analytic ideal, then the ideal $\J$ produced in the proof of Proposition \ref{FDDprop1} is also analytic. Therefore, by combining this observation with continuity result from \cite{RKS}, we see that projections $P_n\colon X \to X_n$ given by FDD are also continuous. 
\end{remark}

We may also inverse the above reasoning -- clearly, spaces with Schauder bases are FDD, but the ideal case is not clear. However, we may solve the special case of some ideals.

\begin{proposition}
\label{FDDprop2}
Let $X$ be a Banach space. If there is an $\I$-Schauder basis of $X$ for some non-tall ideal $\I$ on $\omega$, then $X$ has FDD.
\end{proposition}

\begin{proof}
Let $\hat{\w{a}}$ be the $\I$-Schauder basis of $X$. Since $\I$ is non-tall, there is an infinite $A\subseteq\omega$ such that $\I\restriction A=\Fin\restriction A$. Let $(n_k)$ be the increasing enumeration of the set $A$. Define $X_0=\on{span}\{\w{a}_{0}, \w{a}_{1}, \ldots, \w{a}_{n_{0}}\}$ and $X_k= \on{span}\{\w{a}_{n_{k-1}+1}, \w{a}_{n_k +2}, \ldots, \w{a}_{n_{k}}\} $ for all $k\in\omega\setminus\{0\}$. We claim that $(X_k)$ witnesses that $X$ has FDD. Fix $\w{x}\in X$ and set $\w{x}_0=\sum_{i=0}^{n_{0}}\w{a}^\star_i(\w{x})\w{a}_i\in X_0$ and $\w{x}_k=\sum_{i=n_{k-1}+1}^{n_{k}}\w{a}^\star_i(\w{x})\w{a}_i\in X_k$. Let $\varepsilon>0$ be arbitrary. Since $\hat{\w{a}}$ is an $\I$-Schauder basis of $X$, we have $A_\varepsilon=\left\{j\in \omega: \left\|\w{x}-\sum_{i=0}^{j} \w{a}^\star_i(\w{x})\w{a}_i\right\|_X>\varepsilon \right\}\in\I$. Denote $F=A_\varepsilon\cap A$ and observe that $F\in\Fin$, as $\I\restriction A=\Fin\restriction A$. We have:
\[
\left\{k\in\omega: \left\|\w{x}-\sum_{i=0}^k \w{x}_i\right\|_X>\varepsilon \right\}=\left\{k\in\omega: \left\|\w{x}-\sum_{i=0}^{n_k} \w{a}^\star_i(\w{x})\w{a}_i\right\|_X>\varepsilon \right\}=\left\{k\in\omega:n_k\in F\right\}\in\Fin.
\]
Moreover, such a decomposition is unique, since the scalars $\w{a}^\star_n(\w{x}))$ are unique.
\end{proof}

\begin{theorem}
The following are equivalent for every Banach space $X$:
    \begin{itemize}
        \item[(a)] $X$ has FDD;
        \item[(b)] $X$ has $\I$-FDD for some non-tall ideal $\I$ on $\omega$;
        \item[(c)] $X$ has $\I$-Schuader basis for some non-tall ideal $\I$ on $\omega$. 
    \end{itemize}
\end{theorem}

\begin{proof}
(a)$\implies$(b) is obvious, (b)$\implies$(c) follows from Proposition \ref{FDDprop1} and (c)$\implies$(a) is proved in Proposition \ref{FDDprop2}.
\end{proof}

\begin{theorem}
The following are equivalent for every Banach space $X$:
    \begin{itemize}
        \item[(a)] $X$ has $\I$-FDD for some ideal $\I$ on $\omega$;
        \item[(b)] $X$ has $\I$-Schauder basis for some ideal $\I$ on $\omega$. 
    \end{itemize}
\end{theorem}

\begin{proof}
The implication (b)$\implies$(a) is obvious, while (a)$\implies$(b) follows from Proposition \ref{FDDprop1}.
\end{proof}

\section{Normed vector spaces}\label{S:normed}

In this Section we will show that for every ideal $\I$ on $\omega$ there are a normed vector space $X$ and $\hat{\w{a}}$ such that $\CR(X,\hat{\w{a}})=\I$ (ideals $\CR(X,\hat{\w{a}})$ were defined only for Banach spaces $X$, but this definition can be naturally extended).

\begin{theorem}
If $\I$ is any ideal on $\omega$, then there are a normed vector space $X$ and $\hat{\w{a}}\in(X\times X^\star)^\omega$ such that $\CR(X,\hat{\w{a}})=\I$.
\end{theorem}

\begin{proof}
Let $\J=\{A\subseteq\omega: A\setminus\{0\}=B+1\text{ for some }B\in\I\}$, where $B+1=\{n+1:n\in B\}$. Note that $\J$ is an ideal on $\omega$. Consider $c_{\J}=\{\w{x}\in c_0:\{n\in\omega:x_n\neq 0\}\in\J\}$. Clearly, it is a vector space normed by the supremum norm. 

Using Lemma \ref{lem:any_sequence}, find $\hat{\w{a}}$ with $\w{a}_n\in c_{00}$ (so also $\w{a}_n\in c_{\J}$) for all $n\in\omega$, which is simple over $\hat{\w{e}}$ with witnesses $D=\omega$, $h=\on{id}+1$ and some $(\w{b}_{n+1})$ such that $\|\w{b}_{n+1}\|_{c_0}=n+1$. 

By Lemma \ref{general-lem1} and Remark \ref{remark-general-lem1}, $\CR(X,\hat{\w{a}})$ is the ideal generated by $\Fin$ and all sets of the form:
\[
A_{\w{x}}=\left\{n\in \omega: |\w{e}^\star_{n+1}(\w{x})|\geq\frac{1}{n+1}\right\},
\]
for $\w{x}\in c_{\J}$. We will show that $\CR(X,\hat{\w{a}})=\I$. 

$\CR(X,\hat{\w{a}})\subseteq\I$: Let $\w{x}\in c_{\J}$. Then $A_{\w{x}}\subseteq \{n\in\omega:x_{n+1}\neq 0\}$, so $A_{\w{x}}+1\subseteq \{n\in\omega:x_{n}\neq 0\}\in\J$. Hence, $A_{\w{x}}\in\I$.

$\CR(X,\hat{\w{a}})\supseteq\I$: Let $A\in\I$. Define:
\[
    x_n =
    \begin{cases*}
    \frac{1}{n},     & if $n-1\in A$, \\
    0, & otherwise.  \\
    \end{cases*}
  \]
Observe that $\left\{n\in\omega:x_{n}\neq 0\}=\{n\in\omega:n-1\in A\right\}=A+1\in\J$, so $\w{x}\in c_{\J}$. Moreover, $A_{\w{x}}=\{n\in\omega:|x_{n+1}|\geq\frac{1}{n+1}\}=A$.
\end{proof}

\section{Questions}
We conclude our paper with a list of open questions concerning ideal bases. First of them is somehow motivated by Section \ref{S:avoidingC}.

\begin{question}
    Suppose that $X\in \{c_0, \ell_p\}$ and $(\w{a}_n)\subseteq X$ is such that for every $\w{x}\in X$ there exists a unique sequence of scalars $(\alpha_{n})$ such that for every $k \in \omega$ we have $\w{x}(k)=\sum_{n=0}^\infty \alpha_{k} \left(\w{a}_n\right) (k)$. Is there some ideal $\I$, which makes it an $\I$-Schauder basis of $X$? 
\end{question}

Recall that by Theorem \ref{T:allP}, for every non-pathological analytic $P$-ideal $\I$ there are a Banach space $X\subseteq\mathbb{R}^\omega$ and $\hat{\w{a}}$ simple over $\hat{\w{e}}$ such that $\CR(X,\hat{\w{a}})=\I$.

\begin{question}
Let $\I$ be an analytic $P$-ideals. Are there a Banach space $X$ and $\hat{\w{a}}$ such that $\CR(X,\hat{\w{a}})=\I$? 
\end{question}

By \cite[Theorem B]{RKS} we know that each critical ideal in a Banach space is necessarily analytic, provided that coordinate functionals are continuous. This result suggest to ask some other similar questions. Note that all critical ideals in Banach spaces, which we present in this paper, are P-ideals.

\begin{question}
\label{q1}
Is $\CR(X,\hat{\w{a}})$ necessarily a P-ideal, provided that $X$ is a Banach space?
\end{question}

Note that every analytic P-ideal is $\bf{F_{\sigma\delta}}$ (see \cite[Lemma 1.2.2 and Theorem 1.2.5]{Farah}). Thus, a positive answer to Question \ref{q1} combined with the above mentioned result \cite[Theorem B]{RKS} would show that each critical ideal in a Banach space is necessarily $\bf{F_{\sigma\delta}}$, provided that coordinate functionals are continuous.

\cite[Theorem 1.18]{ARTT} ensures that any critical ideal may be enlarged to a Borel ideal, provided that coordinate functionals are continuous. However, the following question remains open.

\begin{question}
May a critical ideal in some Banach space be analytic, but not Borel?
\end{question}

Next question is also motivated by \cite[Theorem 1.18]{ARTT} (see also \cite[Question 2]{RKS}).

\begin{question}
Is there some Borel class $\Gamma$ such that for every ideal Schauder basis $\hat{\w{a}}$ there exists an ideal $\I\in\Gamma$ such that $\hat{\w{a}}$ is an $\I$-Schauder basis? If the answer is positive, then what is the smallest such $\Gamma$?
\end{question}

Note that a positive answer to Question \ref{q1} would mean that $\Gamma=\bf{F_{\sigma\delta}}$ works for all ideal Schauder bases with continuous coordinate functionals.

The next two questions seem to be most important (and probably hardest) problems in this field:

\begin{question}\cite[Question 1]{RKS}
Are coordinate functionals necessarily continuous for all ideal bases?     
\end{question}

\begin{question}\label{Q:allspaces}
Let $X$ be a separable Banach space. Does $X$ necessarily contain an $\I$-Schauder basis for some nontrivial ideal $\I$?
\end{question}

In view of \cite[Example 1.14]{ARTT} answer for the following question could be viewed as a step towards answering Question \ref{Q:allspaces}.

\begin{question}
Is there a Banach space without FDD, yet equipped with $\I$-Schauder basis for some nontrivial ideal $\I$?
\end{question}

Section \ref{S:normed} suggest the following version of the Question \ref{Q:allspaces}.

\begin{question}
Let $X$ be a separable normed space. Does $X$ necessarily admit an $\I$-Schauder basis for some nontrivial ideal $\I$?
\end{question}

\bibliographystyle{plain}
%\bibliography{literature}

\end{document}